\documentclass[10pt, notitlepage]{article}
\usepackage{amsmath, amssymb, pgfplots,yfonts, setspace, enumitem,titlesec}
\usepackage[toc,page]{appendix}
\usepackage{hyperref}
\usepackage{pst-node}
\usepackage{tikz-cd} 
\usepackage[margin=1.5in]{geometry}
\usepackage{cite}

\def\tr{\mathop{\hbox{tr}}}

\def\id{\mathop{\hbox{id}}}
\def\i{{\bf i}}

\def\Span{\mathop{\hbox{span}}}

\def\Pol{\mathrm{Pol}}

\def\fB{\mathcal{B}}
\def\fH{\mathcal{H}}

\def\fF{\mathcal{F}}

\def\Z{\mathbb{Z}}
\def\C{\mathbb{C}}

\def\H{\mathbb{H}}

\def\G{\mathbb{G}}

\def\wG{\widehat{G}}

\usepackage{mathtools}

\makeatletter
\DeclareRobustCommand\widecheck[1]{{\mathpalette\@widecheck{#1}}}
\def\@widecheck#1#2{%
    \setbox\z@\hbox{\m@th$#1#2$}%
    \setbox\tw@\hbox{\m@th$#1%
       \widehat{%
          \vrule\@width\z@\@height\ht\z@
          \vrule\@height\z@\@width\wd\z@}$}%
    \dp\tw@-\ht\z@
    \@tempdima\ht\z@ \advance\@tempdima2\ht\tw@ \divide\@tempdima\thr@@
    \setbox\tw@\hbox{%
       \raise\@tempdima\hbox{\scalebox{1}[-1]{\lower\@tempdima\box
\tw@}}}%
    {\ooalign{\box\tw@ \cr \box\z@}}}
\makeatother
\usepackage[utf8]{inputenc}
\usepackage[english]{babel}
\usepackage{amsthm}

\theoremstyle{plain}
\newtheorem{thm}{Theorem}[section]
\newtheorem{lem}[thm]{Lemma}
\newtheorem{prop}[thm]{Proposition}
\newtheorem{cor}[thm]{Corollary}
\newtheorem{ques}[thm]{Question}

\theoremstyle{definition}
\newtheorem{defn}[thm]{Definition}
\newtheorem{rem}[thm]{Remark}

\makeatletter
\newcommand*{\toccontents}{\@starttoc{toc}}
\makeatother
\titleformat{\chapter}
  {\normalfont\bfseries}{\thechapter}{1em}{}
\titlespacing*{\chapter}{0pt}{3.5ex plus 1ex minus .2ex}{2.3ex plus .2ex}
\pgfplotsset{compat=1.17}
\begin{document}
\title{On Amenable and Coamenable Coideals}

\author{Benjamin Anderson-Sackaney}
\maketitle

\begin{abstract}
    We study relative amenability and amenability of a right coideal $\widetilde{N}_P\subseteq \ell^\infty(\mathbb{G})$ of a discrete quantum group in terms of its group-like projection $P$. We establish a notion of a $P$-left invariant state and use it to characterize relative amenability. We also develop a notion of coamenability of a compact quasi-subgroup $N_\omega\subseteq L^\infty(\widehat{\mathbb{G}})$ that generalizes coamenability of a quotient as defined by Kalantar, Kasprzak, Skalski, and Vergnioux, where $\widehat{\mathbb{G}}$ is the compact dual of $\mathbb{G}$. In particular, we establish that the coamenable compact quasi-subgroups of $\widehat{\mathbb{G}}$ are in one-to-one correspondence with the idempotent states on the reduced $C^*$-algebra $C_r(\widehat{\mathbb{G}})$. We use this work to obtain results for the duality between relative amenability and amenability of coideals in $\ell^\infty(\mathbb{G})$ and coamenability of their codual coideals in $L^\infty(\widehat{\mathbb{G}})$, making progress towards a question of Kalantar et al{.}.
\end{abstract}

{\flushleft\bf Acknowledgements}: I thank Nico Spronk, Brian Forrest, and Michael Brannan for the supervision and support over the fruition of this project. I would like to thank Nico and Michael especially for their careful reading of this article and helpful comments. I would also like to thank Nicholas Manor for useful discussions at the onset of this project, and especially for introducing me to relative amenability of subgroups. The author was supported by a QEII-GSST scholarship and the ANR project ANR-19-CE40-0002. This work was completed as part of the doctoral thesis of the author.

\begin{center}
\bfseries{Contents}
\end{center}
\toccontents
\section{Introduction}
Understanding the tracial states of $C^*$-algebras and simplicity of $C^*$-algebras are problems of interest to operator algebraists (eg. in classification theory). For a discrete group $G$, whenever $C_r(\widehat{G})$ has the unique trace property (which would be the Haar state), then the traces are well-understood: they are comprised of the Haar state alone. When studying these properties of the reduced $C^*$-algebras of groups, an important class of traces to consider are the indicator functions (which are idempotent states) $1_N\in C_r(\widehat{G})^*\subseteq \ell^\infty(G)$, where $N$ is a normal and amenable subgroup of $G$. Besides the Haar state, a distinguished example is $1_{R_a(G)}$ where $R_a(G)$ is the {\bf amenable radical} of $G$, the largest amenable normal subgroup. More precisely, it was achieved in \cite{BKKO17, KK14} that $C_r(\widehat{G})$ has a unique trace if and only if $R_a(G) = \{e\}$ and if $C_r(\widehat{G})$ is simple, then it has a unique trace.

More generally, the idempotent states on the universal $C^*$-algebra $C_u(\widehat{G})$ are exactly those of the form $1_H$ where $H$ is a subgroup of $G$. The traces are those idempotent states $1_H$ where $H$ is normal. The universal idempotent states of a locally compact quantum group $\G$ (the idempotent states on the universal $C^*$-algebra $C_u(\widehat{\G})$) have received a lot of attention in the literature lately (see \cite{SS11, SS16, Soltan, KK17, FST13, K18}). More specifically, their connection to certain group-like projections and compact quasi-subgroups has been established. In particular, the compact quasi-subgroups of a locally compact quantum group $\G$ are in one-to-one correspondence with the universal idempotent states on (see Section $4.1$ for more on compact quasi-subgroups of compact quantum groups). As far as we can tell, aside from the results in \cite{Soltan} concerning normal idempotent states (normal idempotent states on $L^\infty(\G)$), the reduced idempotent states (idempotent states on the reduced $C^*$-algebra $C_r(\G)$) have been left untouched.

Kalantar, Kasprzak, Skalski, and Vergnioux \cite{KKSV22} coined the notion of a coamenable right coideal of quotient type for compact quantum groups. This notion generalizes to compact quasi-subgroups, which we prove has the following characterization.\\
~\\
{\bf Theorem \ref{reducedIdempState}} {\it Let $\G$ be a discrete quantum group and $N_\omega$ a compact quasi-subgroup of $\widehat{\G}$. We have that $N_\omega$ is coamenable if and only if $\omega\in M^r(\widehat{\G}) = C_r(\widehat{\G})^*$.}\\

See sections $4.1$ and $4.2$ for more.
\begin{rem}
    One advantage of the characterization of coamenability in Theorem \ref{reducedIdempState} is that it immediately generalizes to locally compact quantum groups.
\end{rem}
In this work, we establish some basic theory on the reduced idempotent states of compact quantum groups. In particular, as a simple consequence of Theorem \ref{reducedIdempState}, we obtain the following (which is probably known to experts).\\
~\\
{\bf Corollary \ref{redcenidemps}} {\it Let $\G$ be a discrete quantum group. There is a one-to-one correspondence between the amenable quantum subgroups of $\G$ and the central idempotent states on $C_r(\widehat{\G})$. The tracial central idempotent states on $C_r(\widehat{\G})$ are in one-to-one correspondence with the amenable normal quantum subgroups of $\G$ for which their quotients are unimodular.}\\

In light of the exposition in the first paragraph of this section, this work represents a step towards understanding the reduced idempotent states of a compact quantum group, which is a gap for understanding the unique trace property.

Kalantar et al. also coined the notion of a relatively amenable coideal of a discrete quantum group, and proved that relative amenability of a coideal of quotient type $\ell^\infty(\G/\H)$ is equivalent to amenability of $\H$ \cite[Theorem 3.7]{KKSV22}. So, in this case, we have a connection between relative amenability and coamenability of coideals. In this work we coin the notion of an amenable coideal of a discrete quantum group. These notions of relative amenability and amenability (see Section $2.2$) are motivated by their equivalence with relative amenability and amenability of a closed subgroup of a locally compact group respectively \cite{Cap}. In the case of a discrete group $G$, amenability and relative amenability are equivalent, and have the following characterization.
\begin{thm}\cite{D78,AD03, R71, Cap}\label{ClassicalThm}
    Let $G$ be a discrete group and $H$ a subgroup. The following are equivalent:
    \begin{enumerate}
        \item $H$ is amenable;
        \item $\ell^\infty(G/H)$ is amenable;
        \item $\ell^\infty(G/H)$ is relatively amenable;
        \item $J^1(G,H) = \ell^\infty(G/H)_\perp$ has a bounded right approximate identity (brai);
        \item $J^1(G,H)$ has a brai in $\ell^1_0(G) = \{f\in \ell^1(G) : \int_G f = 0\}$;
        \item $J^1(G,H)$ has a brai in $\ell^1_0(H)$.
    \end{enumerate}
\end{thm}
When $G$ is a locally compact group, conditions 1., 2., 4., and 6. are equivalent, and 2. and 5. are equivalent (cf. \cite{Cap}). It is unknown if relative amenability and amenability coincide for locally compact groups in general. We note, however, that the techniques are trivialized in the discrete setting. In the discrete quantum group setting, where we replace subgroups with right coideals, we have to work significantly harder to reach similar results. Both the fact that quantum groups exist only as virtual objects, and that coideals have no underlying closed quantum subgroup (cf. \cite{Daws}) each introduce barriers of their own. For the coideals that are of quotient type, with Corollary \ref{qsubgroupsAmen} we have an analogue of Theorem \ref{ClassicalThm} for discrete quantum groups.

Essential to Caprace and Monod's work on (relative) amenability is the notion of an $H$-invariant state on $\ell^\infty(G)$. While an $\H$-invariant state on $\ell^\infty(\G)$ is a coherent notion (see Section $4.4$), coideals in general are not necessarily quotients by quantum subgroups. Every coideal $N$ of a discrete quantum group, however, can be assigned a group-like projection $P$ such that
$$N = \widetilde{N}_P := \{x\in\ell^\infty(\G) : (1\otimes P)\Delta_\G(x) = x\otimes P\}$$
(see the exposition following Definition \ref{grpprojDef}). Then, to get around this obstruction, we develop a notion of a $P$-invariant state, where we say $m\in \ell^\infty(\G)^*$ is {\bf $P$-left invariant} if either $(fP)*m = f(P)m$ or $(Pf)*m = f(P)m$ holds for every $f\in\ell^1(\G)$.

Above, we are using the predual action of $\ell^\infty(\G)$ on $\ell^1(\G)$:
$$(xf)(y) = f(yx) ~ \text{and} ~ (fx)(y) = f(xy).$$
Then we prove the following.\\
~\\
{\bf Theorem \ref{RelAmenPInv}} {\it Let $\G$ be a discrete quantum group and $\widetilde{N}_P$ a right coideal with group-like projection $P$. We have that $\widetilde{N}_P$ is relatively amenable if and only if there exists a $P$-left invariant state $\ell^\infty(\G)\to\C$.}\\

Given a group-like projection $P$, we demonstrate that the weak$^*$ closed right invariant subspaces
$$\{x\in\ell^\infty(\G) : (1\otimes P)\Delta_\G(x)(1\otimes P) = x\otimes P \} =: M_P\supseteq \widetilde{N}_P$$
are essential to amenability (see Section $3$). The link to relative amenability of $\widetilde{N}_P$ is observed with the following result.\\
~\\
{\bf Theorem \ref{CondRelAmenMP}} {\it Let $P$ be a group-like projection. We have that $M_P$ is amenable if and only if there exists a state $m : \ell^\infty(\G)\to \C$ such that $m(P) \neq 0$ and $(PfP)*m = f(P)m$ for all $f\in\ell^1(\G)$.}\\

With this analogue of an $H$-invariant state for coideals of discrete quantum groups, we obtain an analogue of Theorem \ref{ClassicalThm} for the subspaces $M_P$ and their associated closed left ideals $J^1(M_P) := (M_P)_\perp := \{f\in \ell^1(\G) : f|_{M_P} = 0\}$.\\
~\\
{\bf Theorem \ref{IdealsAmen}} {\it Assume $\G$ is discrete and $P$ is a group--like projection. TFAE:
\begin{enumerate}
        \item $M_P$ is amenable;
        \item $J^1(M_P)$ admits a brai;
        \item $J^1(M_P)$ admits a brai in $\{P\}_\perp$.
\end{enumerate}}
Quantum group duality gives us a one-to-one correspondence between right coideals of $L^\infty(\G)$ and right coideals of $L^\infty(\widehat{\G})$ via their codual coideals (see Section $4.2$). In particular, given an idempotent state $\omega$ and its compact quasi-subgroup $N_\omega\subseteq L^\infty(\G)$, we identify its codual coideal $\widetilde{N_\omega}\subseteq L^\infty(\widehat{\G})$. For a locally compact group $G$, this duality is the identification of the coideals $L^\infty(G/H)$ and $VN(H) = L^\infty(\widehat{H})$.

In general it is not that hard to show coamenability of a LCQG $\G$ implies amenability of its dual $\widehat{\G}$. It is a highly non-trivial result of Tomatsu \cite{Toma} that a compact quantum group $\widehat{\G}$ is coamenable only if the discrete quantum group $\G$ is amenable, generalizing the case of compact and discrete Kac algebras due to Ruan \cite{R96}, which generalizes Leptin's theorem from the classical discrete setting. In the context of compact quasi-subgroups of quotient type and their codual coideals, Kalantar et al{.} posed the following question.
\begin{ques}\label{MainQuestion}\cite{KKSV22}
    Let $\G$ be a discrete quantum group. Let $\widehat{\H}$ be a closed quantum subgroup of $\widehat{\G}$. Is it true that $L^\infty(\widehat{\G}/\widehat{\H})$ is coamenable if and only if $\ell^\infty(\H)$ is relatively amenable?
\end{ques}
This questions extends to compact quasi-subgroups of compact quantum groups in the following manner: is it true that $N_\omega\subseteq L^\infty(\widehat{\G})$ is coamenable if and only if its codual coideal $\widetilde{N_{P_\omega}}$ is relatively amenable?

We make progress for the compact quasi-subgroup version of Question \ref{MainQuestion}. More specifically, we have the converse when we use amenability of the subspace $M_P$ instead of relative amenability of $\widetilde{N}_P$.\\
~\\
{\bf Theorem \ref{biglemmaconverse}} {\it Let $\G$ be a discrete quantum group and $N_\omega\subseteq L^\infty(\widehat{\G})$ a compact quasi-subgroup with $P =\lambda_{\widehat{\G}}(\omega)$. If $N_\omega$ is coamenable then $M_P$ is amenable.}\\

Our progress for the forwards direction is with Lemma \ref{BigLemma}, which is specialized to the case where $P$ is central. A consequence of the above theorem is a full generalization of Caprace and Monod's characterization of amenability and relative amenability for quantum subgroups of discrete quantum groups (see Corollary \ref{qsubgroupsAmen}). Recall that Kalantar et al. have proved that amenability of a quantum subgroup $\H\leq \G$ of a DQG $\G$ is equivalent to relative amenability of $\ell^\infty(\G/\H)$. In particular, we have established the following.\\
~\\
{\bf Corollary \ref{RelAmenIffAmenQS}} {\it Let $\G$ be a DQG and $\H\leq \G$ a quantum subgroup. Then $\ell^\infty(\G/\H)$ is relatively amenable if and only if it is amenable.}\\

The remainder of the paper is organized as follows: in Section $3$, for a group-like projection $P$ we develop a notion of a $P$-invariant state and relate it to relative amenability of $\widetilde{N}_P$ (Theorem \ref{RelAmenIffAmen}). We achieve similar characterizations of both relative amenability and amenability of the subspaces of the form $M_P$. With these characterizations in hand, we are able to establish a version of Theorem \ref{ClassicalThm} (2. $\iff$ 4. $\iff$ 6.) (Theorem \ref{IdealsAmen}).

In Section $4$ we shift gears towards compact quantum groups and their right coideals, with special attention to their compact quasi-subgroups. We prove that a compact quasi-subgroup $N_\omega$ is coamenable if and only if the associated idempotent state $\omega$ factors through the reduced $C^*$-algebra (Theorem \ref{reducedIdempState}). We then classify the central idempotent states on $C_r(\widehat{\G})$ (Theorem \ref{redcenidemps}). Finally, we prove an anlaogue of Theorem \ref{ClassicalThm} for discrete quantum groups and their quantum subgroups.
\begin{rem}
    Kalantar et al{.} achieved an analogue of (some of) Theorem \ref{ClassicalThm} (1. $\iff$ 3.) for discrete quantum groups in \cite[Theorem 4.7]{KKSV22}. In Section 4, among other things, we present an alternative proof of this same result.
\end{rem}

\section{Preliminaries on Coideals of DQGs}
\subsection{Discrete Quantum Groups}
The notion of a quantum group we will be using is the von Neumann algebraic one developed by Kustermans and Vaes \cite{KV00}. A {\bf locally compact quantum group} (LCQG) $\G$ is a quadruple $(L^\infty(\G), \Delta_\G, h_L, h_R)$ where $L^\infty(\G)$ is a von Neumann algebra; $\Delta_\G : L^\infty(\G) \to L^\infty(\G)\overline{\otimes}L^\infty(\G)$ a normal unital $*$--homomorphism satisfying $(\Delta_\G\otimes\id)\circ \Delta_\G = (\id\otimes\Delta_\G)\circ \Delta_\G$ (coassociativity); and $h_L$ and $h_R$ are normal semifinite faithful weights on $L^\infty(\G)$ satisfying
$$h_L(f\otimes \id)(\Delta_\G(x)) =f(1)h_L(x), ~ f\in L^1(\G), x\in \mathcal{M}_{h_L} ~ \text{(left invariance)}$$
and
$$h_R(\id\otimes f)(\Delta_\G(x)) = f(1)h_R(x), ~f\in L^1(\G), x\in \mathcal{M}_{h_R} ~ \text{(right invariance)},$$
where $\mathcal{M}_{h_L}$ and $\mathcal{M}_{h_R}$ are the set of integrable elements of $L^\infty(\G)$ with respect to $h_L$ and $h_R$ respectively. We call $\Delta_\G$ the {\bf coproduct} and $h_L$ and $h_R$ the {\bf left and right Haar weights} respectively, of $\G$. The predual $L^1(\G) := L^\infty(\G)_*$ is a Banach algebra with respect to the product $f*g:= (f\otimes g)\circ\Delta_\G$ known as {\bf convolution}. This naturally provides us with left and right module actions on $L^\infty(\G)$, realized by the equations
$$f*x = (\id\otimes f)(\Delta_\G(x)) ~ \text{and} ~ x*f = (f\otimes\id)(\Delta_\G(x)).$$

Using $h_L$, we can build a GNS Hilbert space $L^2(\G)$ in which $L^\infty(\G)$ is standardly represented. There is a unitary $W_\G\in L^\infty(\G)\overline{\otimes}\fB(L^2(\G))$ such that $\Delta_\G(x) = W_\G^*(1\otimes x)W_\G$. The unitary $W_\G$ is known as the {\bf left fundamental unitary} of $\G$ respectively. The {\bf left regular representation} is the representation
$$\lambda_\G : L^1(\G)\to \fB(L^2(\G)), ~ f\mapsto (f\otimes\id)W_\G.$$
There is a dense involutive subalgebra $L^1_\#(\G)\subseteq L^1(\G)$ that makes $\lambda_\G|_{L^{\#}(\G)}$ a $*$-representation. We denote the von Neumann algebra $L^\infty(\widehat{\G}) = \lambda_\G(L^1(\G))''$. There exists a LCQG $\widehat{\G} = (L^\infty(\widehat{\G}),\Delta_{\widehat{\G}}, \widehat{h_L}, \widehat{h_R})$, where $\Delta_{\widehat{\G}}$ is implemented by $W_{\widehat{\G}} = \Sigma(W_\G)^*$, where $\Sigma : a\otimes b\mapsto b\otimes a$ is the flip map. Pontryagin duality holds: $\hat{\hat{\G}} = \G$.

A {\bf discrete quantum group (DQG)} is a LCQG $\G$ where $L^1(\G) ~(= \ell^1(\G))$ is unital (cf. \cite{R08}, \cite{Wor}). We denote the unit by $\epsilon_\G$, and it satisfies the counit property:
$$(\epsilon_\G\otimes\id)\circ \Delta_\G = \id = (\id\otimes\epsilon_\G)\circ \Delta_\G.$$
Equivalently, $\widehat{\G}$ is a {\bf compact quantum group (CQG)}, which is a LCQG where $\widehat{h_L} = \widehat{h_R}\in L^1(\widehat{\G})$ is a state, known as the Haar state of $\widehat{\G}$. When $\G$ is discrete, the irreducible $*$-representations of $L^1(\widehat{\G})$, are finite dimensional, where a $*$-representation on locally compact $\G$ is a representation that restricts to a $*$-representation on $L^1_\#(\G)$. We denote a family of representives of irreducibles on $L^1(\widehat{\G})$ by $Irr(\widehat{\G})$, and for each $\pi\in Irr(\widehat{\G})$ we let $\fH_\pi$ denote the corresponding $n_\pi$-dimensional Hilbert space (cf. Section $4.1$). Then,
$$L^\infty(\G) = \ell^\infty(\G) = \ell^\infty-\bigoplus_{\pi\in Irr(\widehat{\G})}M_{n_\pi}$$
as von Neumann algebras, and so we obtain the spatial decomposition
$$\ell^1(\G) = \ell^1-\bigoplus_{\pi\in Irr(\widehat{\G})}(M_{n_\pi})_*.$$
\begin{rem}\label{CommutativeCase}
    The examples of DQGs where $\ell^\infty(\G)$ is commutative are the discrete groups (cf. \cite{T69}), where if $G$ is a discrete group, then $G = (\ell^\infty(G), \Delta_G, m_L, m_R)$ where $m_L = m_R = h_L = h_R$ are the left and right Haar measures, and $\Delta_G(f)(s,t) = f(st)$.
    
    The DQGs where $\ell^1(\G)$ is commutative are the duals of compact groups, where if $G$ is a compact group, then $\widehat{G} = (VN(G), \Delta_{\widehat{G}}, \widehat{h})$ where $\widehat{h}$ is the Plancherel weight and $\Delta_{\wG}(\lambda_G(s)) = \lambda_G(s)\otimes\lambda_G(s)$.
\end{rem}

\subsection{Invariant Subspaces, Ideals, and Quotients}
For a LCQG $\G$, we say a subpsace $X\subseteq L^\infty(\G)$ is {\bf right invariant} if $X*f\subseteq L^\infty(\G)$ for every $f\in \ell^1(\G)$, i.e., it is a right $L^1(\G)$-submodule of $L^\infty(\G)$. We define {\bf left invariance} analogously on the left, and we say $X$ is {\bf two-sided} if it is both left and right $L^1(\G)$-invariant. From the bipolar theorem, we obtain a one-to-one correspondence between weak$^*$ closed right $L^1(\G)$-modules of $L^\infty(\G)$ and norm closed left $L^1(\G)$-modules of $L^1(\G)$ (i.e., closed left ideals): $X\subseteq L^\infty(\G)$ is right invariant if and only if $$X_\perp := \{f\in L^1(\G) : f(X) = 0\}\subseteq L^1(\G)$$
is a left ideal. We obtain similar remarks if we replace left with right and vice versa, and if we use two-sided instead. We will use the notation $J^1(X) = X_\perp$.

Of special interest are the invariant subalgebras. If $N$ is a right invariant von Neumann subalgebra of $L^\infty(\G)$, then we call $N$ a {\bf right coideal}. It is not hard to see that a von Neumann subalgabra $N\subseteq L^\infty(\G)$ is right invariant if and only if $\Delta_\G(N)\subseteq  L^\infty(\G)\overline{\otimes} N$. This makes $\Delta_\G|_{N}$ a {\bf coaction} of $\G$ on $N$ so that right coideals are {\bf $\G$-spaces} that are embedded in $L^\infty(\G)$ (see \cite{KS12}).

In the discrete setting, the right coideals have a certain form.
\begin{defn}\label{grpprojDef}
    A {\bf right group-like projection} of $\G$ is a self-adjoint projection $P\in L^\infty(\G)$ satsfying
    $$(1\otimes P)\Delta_\G(P) = P\otimes P$$
    and a left group-like projection is a self-adjoint projection satisfying
    $$(P\otimes 1)\Delta_\G(P) = P\otimes P.$$
\end{defn}
It follows from \cite[Proposition 1.5]{K18} and \cite[Corollary 1.6]{K18} that every right coideal of a DQG is of the form $$\widetilde{N}_P := \{x\in \ell^\infty(\G) : (1\otimes P)\Delta_\G(x) = x\otimes P\}$$
for some group-like projection $P$. Indeed, given a right coideal $N\subseteq\ell^\infty(\G)$, since $\widehat{\G}$ is compact, the Haar state is finite and \cite[Proposition 1.5]{K18} and \cite[Corollary 1.6]{K18} automatically apply to the codual of $N$ (generally, one must be careful with the details of these claims as \cite{K18} is using the right regular representation whereas we are using the left regular representation. The distinctions that arise will be relevant in Section $4$). It turns out that $R_\G(P)$, where $R_\G$ is the unitary antipode of $\G$ (see Section $4.1$) is the orthogonal projection onto $L^2(\widetilde{\widetilde{N}_P})\subseteq L^2(\widehat{\G})\cong \ell^2(\G)$, which is the Hilbert space generated by
$$\widetilde{\widetilde{N}_P} = \{x\in L^\infty(\widehat{\G}) : Px = xP\} = L^\infty(\widehat{\G})\cap \{P\}'.$$
Then $R_\G(P)$ generates the left version of $\widetilde{N}_P$ from which we recover $\widetilde{N}_P$ using the unitary antipode (see Remark \ref{QuantumSetsRemark}). We say that the coideals $\widetilde{\widetilde{N}_P}$ and $\widetilde{N}_P$ are {\bf codual} to one another, and $\widetilde{\widetilde{N}_P}$ is often referred to as the codual of $\widetilde{N}_P$ and vice versa.

We note that if $\widetilde{N}_P$ is a two-sided coideal, by which we mean is both left and right $\ell^1(\G)$-invariant, then $P$ is both a left and right invariant group-like projection, and moreover, we have $\widetilde{N}_P = \ell^\infty(\H)$, where $\widehat{\H}$ is a (Woronowicz) closed quantum subgroup of $\widehat{\G}$ (cf. \cite{NY14} for the claim and Section $4.2$ for the definition of a (Woronowicz) closed quantum subgroup).

We will be studying relative amenability and amenability of coideals of discrete quantum groups. The notion of a relatively amenable coideal was first coined in \cite{KKSV22}.
\begin{defn}
    Let $\G$ be a locally compact quantum group. A right coideal $N\subseteq L^\infty(\G)$ is {\bf relatively amenable} if there exists a unital completely positive (ucp) right $L^1(\G)$-module map $L^\infty(\G)\to N$.
\end{defn}
Inspired by the above, we make the following definition.
\begin{defn}
    Let $\G$ be a locally compact quantum group. A right coideal $N\subseteq L^\infty(\G)$ is {\bf amenable} if there exists a surjective ucp right $L^1(\G)$-module projection $L^\infty(\G)\to N$.
\end{defn}
We will also make use of the terms amenability and relative amenability in reference to weak$^*$ closed right invariant subspaces of $\ell^\infty(\G)$ as well.

A notion related to relative amenability and amenability is the following, which we will develop in Section $3$.
\begin{defn}\label{pinvariantstate}
    Let $P$ be a group-like projection. We say $m\in \ell^\infty(\G)^*$ is {\bf $P$-left invariant} if either $(fP)*m = f(P)m$ or $(Pf)*m = f(P)m$ holds for every $f\in\ell^1(\G)$.
\end{defn}
A certain subclass of right coideals are manufactured from closed quantum subgroups. We outline their formulation here for DQGs.
\begin{defn}
    Given DQGs $\G$ and $\H$, we say $\H$ is a {\bf open quantum subgroup} of $\G$ if there exists a normal unital surjective $*$-homomorphism $\sigma_\H : \ell^\infty(\G)\to\ell^\infty(\H)$ such that
    $$(\sigma_\H\otimes\sigma_\H)\circ \Delta_\G = \Delta_\H\circ\sigma_\H.$$
\end{defn}
\begin{rem}
    There is much subtlety to the notion of a closed quantum subgroup in general (cf. \cite{Daws}). Open quantum subgroups are always closed, and the converse follows in the discrete case (cf. \cite{Kal1}). In light of this, we will always refer to open quantum subgroups of DQGs as simply quantum subgroups.
\end{rem}
The quotient space $\G/\H$ is defined as follows: we denote
$$l_\H = (\id\otimes\sigma_\H)\circ \Delta_\G$$
and set
\begin{align*}
    \ell^\infty(\G/\H) &= \{x\in\ell^\infty(\G) : l_\H(x) = x\otimes 1\}.
\end{align*}
The space $\ell^\infty(\G/\H)$ is a right coideal. Any right coideal of the form $\ell^\infty(\G/\H)$ is called a {\bf right coideal of quotient type}. In this special case, we use the notation
$$J^1(\G,\H) = J^1(\ell^\infty(\G/\H)).$$
If we let $1_\H$ be the central support of $\sigma_\H$, then we obtain an injective $*$-homomorphism:
$$\iota_\H : \ell^\infty(\H)\to\ell^\infty(\G),  ~\sigma_\H(x)\mapsto 1_\H x.$$
It turns out $1_\H$ is the group-like projection that generates $\ell^\infty(\G/\H)$, i.e., $N_{1_\H} = \ell^\infty(\G/\H)$. Conversely, every central group-like projection in $\ell^\infty(\G)$ generates the quotient of a quantum subgroup (cf.\cite{Kal1}).

\subsection{Amenability of LCQGs}
A LCQG $\G$ is said to be {\bf amenable} if there exists a right invariant mean, i.e., a state $m\in L^\infty(\G)^*$ such that $m(x*f) = f(1)m(x)$ for all $f\in L^1(\G)$ and $x\in L^\infty(\G)$. Amenability is equivalent to the existence of a left invariant mean and also the existence of a mean that is both left and right invariant (cf. \cite{Des}).

A LCQG is said to be {\bf coamenable} if $L^1(\G)$ admits a bounded approximate identity (bai). It was shown in \cite{Toma} that a DQG $\G$ is amenable if and only if $\widehat{\G}$ is coamenable. For a locally compact $\G$ it is known that coamenability of $\widehat{\G}$ implies amenability of $\G$ in general, however, the converse remains open.

\subsection{Annihilator Ideals}
We introduce certain ideals and subspaces of the $L^1$-algebra of a DQG which turn out to be fundamental to amenability and relative amenability of the right coideals (and left and two-sided coideals).
\begin{rem}
    Before proceeding, we make a technical remark. We obtain an action of $L^1(\G)$ on $L^\infty(\G)^*$ by taking the adjoint of the action of $L^1(\G)$ on $L^\infty(\G)$: we set
    $$\omega*f(x) := \omega(f*x) = \omega(\id\otimes f)(\Delta_\G(x)), ~ f\in L^1(\G), \omega\in L^\infty(\G)^*, x\in L^\infty(\G).$$
    Given von Neumann algebras $N$ and $M$, it is clear that the slice maps $\varphi\otimes\id : N\overline{\otimes}M\to M$ are defined for normal functionals $\varphi\in N_*$. While less clear, it is the case that slice maps are still defined if we drop normality and additionally satisfy $(\varphi\otimes\id)(\id\otimes\Phi) = \Phi(\varphi\otimes\id)$ for any normal ucp map $\Phi : M\to K$ to another von Neumann algebra $K$ (consult \cite{Des} or \cite{Neuf}). Thus we are justified in writing
    $$\omega*f (x) = (\omega\otimes f)(\Delta_\G(x)) = f( x*\omega)$$
    and similarly for actions on the left.
\end{rem}
From now on, we will assume $\G$ is discrete.

For a functional $m\in \ell^\infty(\G)^*$ and $x\in \ell^\infty(\G)$, we will use the notation
$$Inv_L(m) = \{f\in \ell^1(\G) : f*m = f(1)m \}$$
and likewise for $Inv_R(m)$ but for normal functionals acting on the right of $m$. Then we set $Inv(m) = Inv_L(m,x)\cap Inv_R(m,x)$. We will denote
$$Ann_L(m) := \{f\in \ell^1(\G) : f*m = 0\}$$
and by $Ann_R(m)$ and $Ann(m)$ we mean the analogous thing.

Using our above notation, amenability is this: there exists a state $m\in \ell^\infty(\G)^*$ such that $\ell^1(\G) = Inv_L(m)$. It is easily seen that amenability is equivalent to the existence of a state $m\in \ell^\infty(\G)^*$ such that
$$\ell^1_0(\G):= \{f\in \ell^1(\G) : f(1) = 0\} = Ann_L(m).$$
We have that $\ell^1_0(\G)$ is an ideal of codimension one in $\ell^1(\G)$, which means, if there is an invariant state $m\in\ell^\infty(\G)^*$, then
$$Inv_L(m) = \ell^1(\G) = \ell^1_0(\G) + \C\epsilon_\G.$$
The generalization of this relationship is as follows.
\begin{prop}\label{AnnInv2}
    Assume $\G$ is discrete. Let $m\in \ell^\infty(\G)^*$ such that $m(1)\neq 0$. Then $Ann_L(m) + \C\epsilon_\G = Inv_L(m)$.
\end{prop}
\begin{proof}
    First note $\ker(f\mapsto f(1))\cap Ann_L(m) = Ann_L(m)$. To see this, notice $f*m = 0$ implies
    $$0 = f*m(1) = f(1)m(1)$$
    so $f(1) = 0$ because $m(1)\neq 0$. So, if $f\in Ann_L(m)$ and $c\in \C$, then $(f + c\epsilon_\G)(1) = c$. Then
    $$(f+c\epsilon_\G)*m = cm = [(f+c\epsilon_\G)(1)]m.$$
    On the other hand, given $f\in Inv_L(m)$, if $f(1) = 0$, then $f\in Ann_L(m)$ is automatic, and otherwise, $(f - f(1)\epsilon_\G)*m = 0$. So $f = (f - f(1)\epsilon_\G) + f(1)\epsilon_\G$ is the desired decomposition.
\end{proof}
There is well-known a correspondence between bounded linear right $\ell^1(\G)$--module maps $E_\omega : \ell^\infty(\G)\to \ell^\infty(\G)$ and functionals $\omega\in\ell^\infty(\G)^*$ via the assignment $\omega =  \epsilon_\G \circ E_\omega$
where
$$E_\omega(x) := \omega*x.$$
We will call $\omega$ {\bf right idempotent} if $\omega(\omega*x) = \omega(x)$ for all $x\in\ell^\infty(\G)$. To put it another way, $\omega$ is right idempotent if it is idempotent with respect to the left Arens product on $\ell^\infty(\G)^*$ (see, for example \cite{HNR12} for more on the left and right Arens products). There is a right Arens product too, which is not generally equal to the left Arens product because $\mu(\id\otimes\nu) = \nu(\mu\otimes\id)$ does not necessarily hold for every $\mu,\nu\in \ell^\infty(\G)^*$. Note that the left and right Arens product on $C_u(\widehat{\G})^*$ coincide and are equal to the convolution product, hence there is no need to distinguish between right and left idempotency in that context.
\begin{rem}\label{ModMapsAnn}
    \begin{enumerate}
    \item It is the case that $\omega$ is right idempotent exactly when $E_\omega$ is idempotent, and likewise for (complete) positivity and unitality. In particular, $E_\omega$ is a unital completely positive (ucp) projection exactly when $\omega$ is a right idempotent state (see \cite{JNR09, HNR12} for an account of right $L^1(\G)$--module maps in the setting of LCQGs). Clearly, $\omega*\ell^\infty(\G)$ is a norm closed right invariant subspace of $\ell^\infty(\G)$. Whenever $\omega$ is right idempotent state, $E_\omega$ is a ucp projection.

    \item We also point out that the easy general fact
    $$(B + \C\epsilon_\G)^\perp = \ker(\epsilon_\G)\cap B^\perp,$$
    where $B\subseteq\ell^1(\G)$ is a subset, combined with Proposition \ref{AnnInv2} tells us
    $$\overline{\omega*\ell^\infty(\G)}^{wk*}\cap\ker(\epsilon_\G) = Inv_L(\omega)^\perp = \ker(\epsilon_\G)\cap Ann_L(\omega)^\perp$$
    whenever $\omega(1)\neq 0$.
    \end{enumerate}
\end{rem}
Of course, we can use approximation arguments to study the spaces $Inv_L(m)$ and $Ann_L(m)$.
\begin{lem}\label{ApproxInv}
    Let $m\in L^\infty(\G)^*$ be a state. There exists a net of states $(\omega_i)\subseteq L^1(\G)$ such that
    $$f*\omega_i - f(1)\omega_i\to 0$$
    for all $f\in Inv_L(m)\cup Ann_L(m)$.
\end{lem}
\begin{proof}
    The argument follows from the proof of the corresponding statement for left invariant means with minor adjustments (see \cite{Enock}). To elaborate, we obtain a net $(\omega_i)$ such that $f*\omega_i - f(1)\omega_i\to^{wk}0$ from weak density of $\overline{B_1(L^1(\G))_+}$ in $\overline{B_1(L^\infty(\G)^*)_+}$, and we achieve norm convergence from a convexity argument on the space $\prod_{f\in Inv_R(m)\cup Ann_R(m)} L^1(\G)$.
\end{proof}

\section{Amenability and Relative Amenability of Coideals}
We continue to assume $\G$ is a DQG. Recall that the right coideals of $\ell^\infty(\G)$ are of the form $\widetilde{N}_P$ for a group-like projection $P$. We will establish the role $P$ plays in amenability and relative amenability of $\widetilde{N}_P$ as a right coideal.

The following useful lemma is probably well known to experts, but we provide a proof for convenience.
\begin{lem}\label{ProjectionIdentity}
    Assume $\G$ is discrete and $P$ is a group-like projection. Let $m$ be a functional such that $m*\ell^\infty(\G)\subseteq \widetilde{N}_P$. Then $P(m*x) = m(x)P$ for all $x\in \ell^\infty(\G)$.
\end{lem}
\begin{proof}
    First, notice for $x\in \widetilde{N}_P$ that
    $$P(\epsilon_\G(x)) = \epsilon_\G(x)\otimes P = (\epsilon_\G\otimes\id)(1\otimes P)(\Delta_\G(x)) = 1\otimes  Px= Px.$$
    We point out that above fact appears explicitly in the literature (see the proof of \cite[Theorem 3.1]{FK17}, however, we should emphasize that the claim there does not involve the counit in the general case). Now, for $x\in \ell^\infty(\G)$, by assumption $m*x\in \widetilde{N}_P$, so,
    $$P(m*x) = P\epsilon_\G(m*x) = m(x)P.$$
\end{proof}
As mentioned in the introduction, in classical setting of a discrete group $G$, amenability of a subgroup $H\leq G$ is equivalent to relative amenability of $\ell^\infty(G/H)$, which, in turn, is equivalent to the existence of an $H$-invariant state on $\ell^\infty(G)$. The following theorem establishes an analogue of an $H$-invariant state for coideals of DQGs.
\begin{rem}
    We denote the canonical predual action of $\ell^\infty(\G)$ on $\ell^1(\G)$ as follows:
    $$xf(y) = f(yx) ~ \text{and} ~ fx(y) = f(xy), ~ x,y\in\ell^\infty(\G), ~f\in\ell^1(\G).$$
\end{rem}
Recall the definition of a $P$-left invariant state in Definition \ref{pinvariantstate}.
\begin{thm}\label{CondRelAmen}
    Let $P$ be a group-like projection and $\omega\in\ell^\infty(\G)^*$. Then $\omega*\ell^\infty(\G)\subseteq \widetilde{N}_P$ if and only if $\omega$ is $P$-left invariant.
\end{thm}
\begin{proof}
    Notice that, given $x\in\ell^\infty(\G)$ and $f\in\ell^1(\G)$,
    \begin{align*}
        (f  P)*\omega(x) = f(P(\omega*x))
    \end{align*}
    so $(f P)*\omega = f(P)\omega$ for all $f\in\ell^1(\G)$ if and only if
    \begin{align}
        (\id\otimes \omega)(P\otimes 1)(\Delta_\G(x)) = P(\omega*x) = P\omega(x). ~ \label{CondRelAmenEq1}
    \end{align}
    So, if we assume $\ell^1(\G) P\subseteq Inv_L(\omega)$ then
    \begin{align*}
        (1\otimes P)\Delta_\G(\omega*x) &= (\id\otimes\id\otimes\omega)(1\otimes P\otimes 1)(\Delta_\G\otimes\id)\circ\Delta_\G(x)
        \\
        &= (\id\otimes\id\otimes\omega)(1\otimes P\otimes 1)(\id\otimes\Delta_\G)\circ\Delta_\G(x)
        \\
        &= (1\otimes P)(\id\otimes\omega)(\Delta_\G(x)) ~ \text{(using \eqref{CondRelAmenEq1})}
        \\
        &= (x*\omega)\otimes P.
    \end{align*}
    Conversely,
    $$\overline{\omega*\ell^\infty(\G)}^{wk*}\cap \{x\in \ell^\infty(\G) : Px = 0\} = \overline{\omega*\ell^\infty(\G)}^{wk*}\cap \ker\epsilon_\G = Inv_L(\omega)^\perp,$$
    where the first equality clearly follows from the more general fact $Px = \epsilon_\G(x)P$, for $x\in \widetilde{N}_P$, as pointed out in Lemma \ref{ProjectionIdentity}, and the second was pointed out in Remark \ref{ModMapsAnn}. We have that $x\in\ell^\infty(\G)$ satisfies $0 = (f P)(x) = f(Px)$ for all $f\in\ell^1(\G)$ if and only if $Px = 0$. So
    $$(\ell^1(\G) P)^\perp = \{x\in\ell^\infty(\G) : Px = 0\}.$$
    Hence $Inv_L(\omega)^\perp \subseteq (\ell^1(\G) P)^\perp$, which implies $\ell^1(\G) P\subseteq Inv_L(\omega)$.
\end{proof}
Notice that the above claims follow through if we replace $\ell^1(\G)P$ with $P\ell^1(\G)$. That is $(fP)*\omega = f(P)\omega$ for all $f\in\ell^1(\G)$ if and only if $(Pf)*\omega = f(P)\omega$ for all $f\in\ell^1(\G)$.

In particular, we have proved the following, where a $P$-left invariant state was defined in Section 2.2.
\begin{thm}\label{RelAmenPInv}
    Let $\G$ be a DQG and $\widetilde{N}_P$ a right coideal with group-like projection $P$. We have that $\widetilde{N}_P$ is relatively amenable if and only if there exists a $P$-left invariant state $\ell^\infty(\G)\to\C$.
\end{thm}
\begin{rem}
    \begin{enumerate}
        \item Take $f, g\in \ell^1(\G)$. An easy calculation shows $(f P)*(g P) = ((f P)*g) P$, which means $\ell^1(\G) P$ is a closed subalgebra of $\ell^1(\G)$.
        
        \item The algebra $\ell^1(\G) P$ was studied in \cite{FSS07} for the dual of a compact group $\widehat{G}$. In that setting, $\ell^1(\widehat{G}) P = A(G/K)$, which is the Fourier algebra of the coset space $G/K$ for a closed subgroup $K$.
    \end{enumerate}
\end{rem}
Given a group--like projection $P$, we will denote the weak$^*$ closed right invariant subspaces
$$M_P := \{x\in\ell^\infty(\G) :  (1\otimes P)\Delta_\G(x)(1\otimes P) = x\otimes P\}\supseteq \widetilde{N}_P.$$
We will also use the notation
$$J^1(M_P) := (M_P)_\perp.$$
These subspaces allow us to establish a relationship between amenability and $P$-invariant states on $\ell^\infty(\G)$. The key property is that $x\mapsto PxP$ is a positive map so that states that are conjugated by $P$ remain positive. This will be indispensible when we relate amenability of $M_P$ with brais on $J^1(M_P)$.
\begin{rem}
    We have been unable to determine whether not we generally have $M_P = \widetilde{N}_P$. If $\H$ is a quantum subgroup of $\G$, since $1_\H$ is central, we have $\ell^\infty(\G/\H) = \widetilde{N}_{1_\H} = M_{1_\H}$.
\end{rem}
\begin{thm}\label{RelAmenIffAmen}
    Assume $\G$ is discrete and $0\neq P$ is a group--like projection. Then $M_P$ is amenable in $\G$ if and only if $M_P$ is relatively amenable via a state $m\in \ell^\infty(\G)^*$ such that $m(P)\neq 0$ (and $m*\ell^\infty(\G)\subseteq M_P$).
\end{thm}
\begin{proof}
    First assume $M_P$ is amenable with right idempotent state $m\in \ell^\infty(\G)^*$ such that $M_P = m*\ell^\infty(\G)$. Assume for a contradiction that $m(P) = 0$. Since $P$ is group-like and generates $M_P$, $P\in M_P$, and so $m*P = P$ by assumption. But then
    $$P = P(m*P) = m(P)P = 0 ~ \text{(Lemma \ref{ProjectionIdentity})}.$$
    which contradicts our assumption.
    
    Now we will prove the converse. We will first see that $\frac{1}{m(P)}(P m P)$ is a right idempotent state. Since $x\mapsto PxP$ is positive, $PmP$ is a positive functional and since $\frac{1}{m(P)}(PmP)$ is unital, it is a state. For right idempotency, take $x\in\ell^\infty(\G)$. Then,
    \begin{align*}
        (\frac{1}{m(P)}(PmP))\left( (\frac{1}{m(P)}(PmP))*x\right) &= \frac{1}{m(P)}m[(\id\otimes \frac{1}{m(P)}m)(P\otimes P)\Delta_\G(x)(P\otimes P)]
        \\
        &= \frac{1}{m(P)^2}m[(P\otimes 1)(\id\otimes m)\Delta_\G(PxP)(P\otimes 1)]
        \\
        &= \frac{1}{m(P)^2}\overbrace{m[P(m*(PxP))P]}^{=m(Pm(PxP)P) ~ \text{(Lemma \ref{ProjectionIdentity})}}
        \\
        &= \frac{1}{m(P)}m(PxP).
    \end{align*}
    For the remainder of the proof we will show $P m P$ satisfies $(P m P)*\ell^\infty(\G) = M_P$, where we replace $m$ with $\frac{1}{m(P)}m$, (so $P m P(1) = 1$). Note that Lemma \ref{ProjectionIdentity} still applies to $m$ after scaling. First, take $x\in M_P$. Then
    \begin{align*}
        P m P*x &= (\id\otimes m)(1\otimes P)\Delta_\G(x)(1\otimes P) = x,
    \end{align*}
    shows $M_P\subseteq (P m P)*\ell^\infty(\G)$. On the other hand, for $x\in \ell^\infty(\G)$,
    \begin{align*}
        &(1\otimes P)\Delta_\G((PmP)*x)(1\otimes P)
        \\
        &= (\id\otimes \id\otimes PmP)(1\otimes P\otimes 1)(\Delta_\G\otimes \id)\circ\Delta_\G(x)(1\otimes P\otimes 1)
        \\
        &= (\id\otimes \id\otimes m)(1\otimes P\otimes P)[(\id\otimes\Delta_\G)\circ\Delta_\G(x)](1\otimes P \otimes P)
        \\
        &= (\id\otimes \id\otimes m)(1\otimes P\otimes 1)[(\id\otimes\Delta_\G)((1\otimes P)\Delta_\G(x)(1\otimes P))](1\otimes P\otimes 1)
        \\
        &= (1\otimes P\otimes 1)(\id\otimes \id\otimes m)(\id\otimes\Delta_\G)[(1\otimes P)\Delta_\G(x)(1\otimes P)](1\otimes P\otimes 1)
        \\
        &= (\id\otimes m\otimes \id)(1\otimes P\otimes 1)(\Delta_\G(x)\otimes P) ~ \text{(Lemma \ref{ProjectionIdentity})}
        \\
        &= (PmP)*x \otimes P.
    \end{align*}
    where we used group-likeness in the third equality. We conclude that $(P m
     P)*x\in M_P$.
\end{proof}

\begin{prop}\label{RelAmenIffAmenMP}
    Let $P$ be a group--like projection and $\omega\in \ell^\infty(\G)^*$. Then $\omega*\ell^\infty(\G)\subseteq M_P$ if and only if $P\ell^1(\G)P\subseteq Inv_L(\omega)$.
\end{prop}
\begin{proof}
    The proof is completely analogous to the proof of Theorem \ref{CondRelAmen}. Indeed, we have $P\ell^1(\G)P\subseteq Inv_L(\omega)$ if and only if $P(\omega*x)P = P\omega(x)$. Then, if $P\ell^1(\G)P\subseteq Inv_L(\omega)$, it is readily checked that $(1\otimes P)\Delta_\G(\omega*x)(1\otimes P) = \omega*x\otimes P$ for every $x\in\ell^\infty(\G)$. Conversely, one may check that $Inv_L(\omega)^\perp = \overline{\omega*\ell^\infty(\G)}^{wk*}\cap \{x\in\ell^\infty(\G) : PxP = 0\}$ as in the proof of Theorem \ref{CondRelAmen} (indeed, it can be checked that $PxP = P\epsilon_\G(x)$ for every $x\in M_P$). Then $f(PxP) = 0$ for every $f\in\ell^1(\G)$ if and only if $PxP = 0$, so $(P\ell^1(\G)P)^\perp = \{x\in \ell^\infty(\G) : PxP = 0\}$. Hence, $Inv_L(\omega)^\perp\subseteq \{x\in\ell^\infty(\G) : PxP = 0\}$.
\end{proof}
As a consequence, we obtain the following characterization of amenability of $M_P$.
\begin{thm}\label{CondRelAmenMP}
    Let $\G$ be a DQG and $P$ be a group-like projection. We have that $M_P$ is amenable if and only if there exists a state $m : \ell^\infty(\G)\to \C$ such that $m(P) \neq 0$ and $(PfP)*m = f(P)m$ for all $f\in\ell^1(\G)$.
\end{thm}
\begin{proof}
    This is a straightforward application of Proposition \ref{RelAmenIffAmenMP} and Theorem \ref{RelAmenIffAmen}
\end{proof}
Now we characterize amenability of $M_P$ in terms of the existence of brais for $J^1(M_P) = (M_P)_\perp$.
\begin{thm}\label{IdealsAmen}
    Assume $\G$ is discrete and $P$ is a group--like projection. TFAE:
    \begin{enumerate}
        \item $M_P$ is amenable;
        \item $J^1(M_P)$ admits a brai;
        \item $J^1(M_P)$ admits a brai in $\{P\}_\perp$, by which we mean, there exists a bounded net $(e_j)\subseteq \{P\}_\perp$ such that $f*e_j - f\to 0$ for all $f\in J^1(M_P)$.
\end{enumerate}
\end{thm}
\begin{proof}
    $(1\implies 2)$ It is a standard argument for a Banach algebra $A$ that admits a bai (in particular, if it is unital), that a closed left ideal $I$ admits a brai if and only if there exists a right $A$-module projection $A\to I^\perp$ (see \cite{Curtis} and \cite[Theorem 7]{Cap}). We apply this argument to the left ideal $J^1(M_P)\subseteq\ell^1(\G)$.
    
    $(2\implies 3)$ follows because $J^1(M_P)\subseteq \{P\}_\perp$. What remains is showing $(3\implies 1)$. To this end, let $(e_i)\subseteq \{P\}_\perp$ be a brai in $J^1(M_P)$ and $e$ a weak$^*$ cluster point. Set $\omega = \epsilon_\G - e$. Notice that for all $f\in J^1(M_P)$, $f*\omega = f*\epsilon_\G - f*e = f-f = 0$, i.e., $Ann_L(\omega)\supseteq J^1(M_P)$, which implies $M_P\supseteq \omega*\ell^\infty(\G)$. Notice also that $\omega(P) = 1\neq 0$.
    
    Using Theorem \ref{CondRelAmenMP}, we have
    $$P \ell^1(\G) P\subseteq Inv_L(\omega).$$
    From here, we point out that $P\ell^1(\G) P$ is spanned by states since the map $f\mapsto P f P$ preserves positive elements. So, take a state $f\in P \ell^1(\G) P$. We can assume $\omega$ is Hermitian since the decomposition $\omega = \Re(\omega) + \i\Im(\omega)$ is unique and we must have $\Re(\omega)(P)\neq 0$ or $\Im(\omega)(P) \neq 0$. Now let $\omega = \omega_+ - \omega_-$ be the Jordan decomposition, uniquely determined so that $||\omega_+|| + ||\omega_-|| = ||\omega||$ (cf. \cite[Theorem 4.2]{Tak}). Then, since $f*\omega_+$ is positive,
    $$||f*\omega_+|| = (f\otimes\omega_+)(\Delta_\G(1)) = \omega_+(1) = ||\omega_+||$$
    and similarly $||f*\omega_-|| = ||\omega_-||$. So, by uniqueness, we must have $f*\omega_+ = \omega_+$ and $f*\omega_- = \omega_-$. Without loss of generality, suppose $\omega_+(P) \neq 0$, so we denote $m = \omega_+$, and what we have shown is $P \ell^1(\G) P\subseteq Inv_L(m)$. From here, Proposition \ref{RelAmenIffAmenMP} and Theorem \ref{CondRelAmenMP} give us amenability of $M_P$.
\end{proof}

\section{Amenability and Coamenability of Coideals}
\subsection{Preliminaries on Coideals of CQGs}
For the moment, let us focus on a compact quantum group $\widehat{\G}$, the dual of a DQG $\G$. For this discussion, we are required to discuss the $C^*$-algebraic formulation of quantum groups in the universal setting (cf. \cite{K01}). For simplicity, we denote the Haar state on $\widehat{\G}$ by $h$. Given a representation $\pi : L^1(\widehat{\G})\to \fB(\fH_\pi)\cong M_{n_\pi}$, there exists an operator $U^\pi\in L^\infty(\widehat{\G})\overline{\otimes}\fB(\fH_\pi)$ such that
$$\pi(f) = (f\otimes \id)U^\pi, ~ f\in L^1(\widehat{\G}).$$
A representation $L^1(\widehat{\G})\to \fB(\fH_\pi)$ is called a $*$-representation if $U^\pi$ is unitary. Representations $\pi$ and $\rho$ are unitarily equivalent if there exists a unitary $U\in M_{n_\pi}$ such that $(1\otimes U^*)U^\pi(1\otimes U) = U^\rho$. We let $[\pi]$ denote the equivalence class of representations unitarily equivalent to $\pi$ and we let $Irr(\widehat{\G})$ be a family of representations with exactly one representative chosen from each equivalence class of irreducible representations. Note that every time we have a representation $\pi$, we are choosing a representative from the equivalence class $[\pi]$. In the instance where $\pi\in Irr(\widehat{\G})$, we write $U^\pi = [u_{i,j}^\pi]_{i,j}$, so $\pi(f) = [f(u_{i,j}^\pi)]_{i,j}$ for $f\in L^1(\widehat{\G})$ and some orthonormal basis (ONB) $\{e_i^\pi\}$ of $\fH_\pi$. We let $\overline{\pi} : L^1(\widehat{\G})\to \fB(\fH_{\overline{\pi}})$ denote the representation where $U^{\overline{\pi}} = [(u_{i,j}^\pi)^*] = \overline{U^\pi}$.

It turns out that every $*$-representation decomposes into a direct sum of irreducibles and the left regular representation decomposes into a direct sum of elements from $Irr(\widehat{\G})$, each with multiplicity $n_\pi$. Consequently, we have $W_{\widehat{\G}} = \bigoplus_{\pi\in Irr(\widehat{\G})} U^\pi\otimes I_{n_{\overline{\pi}}}$ using the identification $L^2(\widehat{\G}) = \oplus_{\pi\in Irr(\widehat{\G})}\fH_\pi\otimes \fH_{\overline{\pi}}$.

For each $\pi\in Irr(\widehat{\G})$, there exists a unique positive, invertible matrix $F_\pi$, satisfying $tr(F_\pi) = tr(F_\pi^{-1}) > 0$ such that
$$((U^\pi)^t)^{-1} = (1\otimes F_\pi) \overline{U^\pi}(1\otimes F_\pi^{-1}).$$
We will say $F_\pi$ is the {\bf $F$-matrix} associated with $\pi$. It was shown in \cite{D10} that a representative $\pi$ and ONB may be chosen so that $F_\pi$ is diagonal. We will not make such a choice here, however, because we will be choosing representatives and ONBs for other reasons.

The left and right Haar weights on $\ell^\infty(\G)$ satisfy the decompositions
$$h_L = \bigoplus_{\pi\in Irr(\widehat{\G})}tr(F_\pi)tr(F_\pi\cdot) ~ \text{and} ~ h_R = \bigoplus_{\pi\in Irr(\widehat{\G})}tr(F_\pi)tr(F_\pi^{-1}\cdot), ~ \text{\cite{Tim}}.$$
Note that our convention for choice of $F_\pi$ follows \cite{D10}. Here $tr(I_{n_\pi}) = n_\pi$ is the normalization of $tr$.
We denote the $*$-algebra
$$\Pol(\widehat{\G}) = \Span\{ u^\pi_{i,j} : 1\leq i,j\leq n_\pi, \pi\in Irr(\widehat{\G}))\}\subseteq L^\infty(\widehat{\G}).$$
It follows that $L^\infty(\widehat{\G}) = \overline{\Pol(\widehat{\G})}^{wk*}$ from Pontryagin duality.

There exists a universal $C^*$-norm $||\cdot||_u$ on $\Pol(\widehat{\G})$. Let $||\cdot||_r$ be the norm on $\fB(L^2(\widehat{\G}))$. We define the unital $C^*$-algebras $C_u(\widehat{\G}) = \overline{\Pol(\widehat{\G})}^{||\cdot||_u}$ and $C_r(\widehat{\G}) = \overline{\Pol(\widehat{\G})}^{||\cdot||_r}\subseteq L^\infty(\widehat{\G})$. The universal property gives us a $C^*$-algebraic coproduct on $C_u(\widehat{\G})$: a unital $*$-homomorphism
$$\Delta^u_\G : C_u(\widehat{\G})\to C_u(\widehat{\G})\otimes_{min}C_u(\widehat{\G})$$
satisfying coassociativity. Likewise, $\Delta^r_{\widehat{\G}} = \Delta_{\widehat{\G}}|_{C_r(\widehat{\G})}$ gives us a $C^*$-algebraic coproduct on $C_r(\widehat{\G})$. We denote $C_u(\widehat{\G})^* = M^u(\widehat{\G})$ and $C_r(\widehat{\G})^* = M^r(\widehat{\G})$, and are known as the {\bf universal and reduced measure algebras} of $\widehat{\G}$ respectively. Similar to the von Neumann algebraic case, the coproduct on $C_u(\widehat{\G})$ induces a product on $M^u(\widehat{\G})$:
$$\mu*\nu(a) = (\mu\otimes\nu)(\Delta(a)), ~ a\in C_u(\widehat{\G}), \mu,\nu \in M^u(\widehat{\G}),$$
making $M^u(\widehat{\G})$ a Banach algebra, where, above, $C_u(\widehat{\G})$ and $M^u(\widehat{\G})$ can be either the universal or reduced versions.

The universal property gives us a unital surjective $*$-homomorphism
$$\Gamma_{\widehat{\G}} : C_u(\widehat{\G})\to C_r(\widehat{\G})$$
which intertwines the coproducts:
$$\Delta_{\widehat{\G}}^r\circ \Gamma_{\widehat{\G}} = (\Gamma_{\widehat{\G}}\otimes\Gamma_{\widehat{\G}})\circ\Delta^u_{\widehat{\G}}.$$
The adjoint of this map induces a completely isometric $*$-homomorphism $M^r(\widehat{\G})\to M^u(\widehat{\G})$ such that $M^r(\widehat{\G})$ is realized as a weak$^*$ closed ideal in $M^u(\widehat{\G})$. A Hahn-Banach argument shows $\overline{L^1(\widehat{\G})}^{wk*} = M^r(\widehat{\G})$, so $L^1(\widehat{\G})$ is a closed ideal in $M^u(\widehat{\G})$ as well.

If we let
$$\ell^1_F(\G) = \oplus_{\pi\in Irr(\widehat{\G})}(M_{n_\pi})_*$$ then $\lambda_\G(\ell^1_F(\G)) = \Pol(\widehat{\G})$. Indeed, let $\delta^\pi_{i,j}\in (M_{n_\pi})_*\subseteq \ell^1(\G)$ be the dual basis element, i.e., the functional which satisfies $\delta_{i,j}^\pi(E_{k,l}^\sigma) = \delta_{i,k}\delta_{j,l}\delta_{\pi,\sigma}$, where $E_{k,l}^\sigma\in M_{n_\sigma}$ is the matrix unit with respect to the ONB $\{e^\sigma_{i,j}\}$ of $\fH_\sigma$ (has $1$ in entry of the $k$-th row and $l$-column and $0$ in every other entry), with which we have the decomposition $[u_{i,j}^\pi] = \sum_{i,j=1}^{n_\pi} u_{i,j}^\pi\otimes E_{i,j}^\pi$. Then
$$\lambda_\G(\delta^\pi_{i,j}) = (\delta_{i,j}^\pi\otimes\id) W_\G = (\delta_{i,j}\otimes\id)\bigoplus_{\pi\in Irr(\widehat{\G})} \Sigma(U^\pi)^* = (u_{j,i}^\pi)^*.$$
The restriction of the coproduct is a unital $*$-homomorphism $\Delta_{\widehat{\G}} : \Pol(\widehat{\G})\to \Pol(\widehat{\G})\otimes \Pol(\widehat{\G})$ that satisfies
$$\Delta_{\widehat{\G}}(u_{i,j}^\pi) = \sum_{t=1}^{n_\pi}u_{i,t}^\pi\otimes u_{t,j}^\pi.$$
The unital $*$-homomorphism
$$\epsilon_{\widehat{\G}} : \Pol(\widehat{\G})\to\C, ~ u_{i,j}^\pi\mapsto \delta_{i,j}$$
extends to a unital $*$-homorphism $\epsilon^u_{\widehat{\G}} : C_u(\widehat{\G})\to \C$ which is the identity element in $M^u(\widehat{\G})$. We have that $\widehat{\G}$ is coamenable if and only if $\epsilon_{\widehat{\G}}^u\in M^r(\widehat{\G})$ if and only if $\G$ is amenable (cf. \cite{R96}).

If for $a\in \Pol(\widehat{\G})$ we let $\hat{a}\in L^1(\widehat{\G})$ denote the functional satisfying $\hat{a}(x) = h(a^*x)$ for $x\in L^\infty(\widehat{\G})$. Then
$$\lambda_{\widehat{\G}}(\widehat{\Pol(\widehat{\G})}) = \bigoplus_{\pi\in Irr(\widehat{\G})}M_{n_\pi} =: c_{00}(\G).$$
In fact, we have
$$E_{i,j}^\pi = \sum_{k=1}^{n_\pi}\frac{1}{\tr(F_\pi)}(F_\pi)_{i,k}\lambda_{\widehat{\G}}(\widehat{u_{k,j}^\pi}).$$

Let $X\subseteq L^\infty(\widehat{\G})$ be a weak$^*$ closed right $L^1(\widehat{\G})$-module. It was established in\cite{AS22} that there exists a hull $E = (E_\pi)_{\pi\in Irr(\widehat{\G})}$, where each $E_\pi\subseteq\fH_\pi$ is a subspace, such that
$$j(E)\subseteq J^1(X)\subseteq I(E)$$
where
$$I(E) = \{ f\in L^1(\widehat{\G}) : \pi(f)(E_\pi) = 0, \pi\in Irr(\widehat{\G})\}$$
and
$$j(E) = I(E)\cap \widehat{\Pol(\widehat{\G})}$$
where $\widehat{a}(b) = h(a^*b)$ for $a, b\in \Pol(\widehat{\G})$. Then
$$L^\infty(\widehat{E})\subseteq X\subseteq j(E)^\perp$$
where $L^\infty(\widehat{E}) = \overline{\Pol(\widehat{E})}^{wk*}$, and
$$\Pol(\widehat{E}) = \{ u_{\xi,\eta}^\pi : \eta\in E_\pi, \xi\in \fH_\pi \},$$
where $u^\pi_{\xi,\eta} = (\id\otimes w_{\eta,\xi})U^\pi$. We will call a sequence of the form $E$ a {\bf closed quantum subset of $\widehat{\G}$}. It was shown in \cite{AS22} that $\overline{j(E)} = I(E)$ for every closed quantum subset $E$ if and only if $f\in \overline{L^1(\widehat{\G})*u}$ for every $f\in L^1(\widehat{\G})$. This property is known as property left $D_\infty$ of $\G$. There are no known examples of DQGs that do not satisfy property left $D_\infty$, even for discrete groups. It is not even known if there are discrete groups without the approximation property but with property left $D_\infty$ (for example, it is unknown if $SL_3(\Z)$ has property left $D_\infty$ (cf. \cite{A20})).

For a LCQG $\G$, recall that the {\bf unitary antipode} is the $*$-antiautomorphism $R_\G : L^\infty(\G)\to L^\infty(\G)$ defined by setting $R_\G(x) = Jx^* J$, where $J : L^2(\G)\to L^2(\G)$ is the modular conjugation for $h_L$. For a CQG $\widehat{\G}$ and $\pi\in Irr(\widehat{\G})$, the unitary antipode satisfies
\begin{align}\label{unitaryAntipode}
    (R_{\widehat{\G}}\otimes\id)U^\pi = (1\otimes F_\pi^{1/2}) (U^\pi)^* (1\otimes F_\pi^{-1/2}).
\end{align}
In general, for locally compact $\G$, the unitary antipode satisfies
$$(R_\G\otimes R_\G)\circ \Delta_\G = \Sigma\Delta_\G\circ R_\G$$
and so it is straightforward to see that if $N$ is a right coideal, then $R_\G(N)$ is a left coideal.

For a DQG $\G$, if $E$ is a closed quantum subset, then
$$R_{\widehat{\G}}(\Pol(\widehat{E})) =  \Span\{u^\pi_{\xi,\eta} : \xi\in E_\pi, \eta\in\fH_\pi, \pi\in Irr(\widehat{\G})\}$$
is left $L^1(\widehat{\G})$-invariant.

For each $\pi\in Irr(\widehat{\G})$, let $P_\pi\in M_{n_\pi}$ be the orthogonal projection onto $E_\pi\subseteq\fH_\pi$. We will denote $P_E = \oplus_{\pi\in Irr(\widehat{\G})}P_\pi\in \ell^\infty(\G)$. Now, $\ell^\infty(\G)P_E$ is a weak$^*$ closed right ideal in $\ell^\infty(\G)$. Conversely, for any weak$^*$ closed right ideal $I$ in $\ell^\infty(\G)$, there is an orthogonal projection $P = \oplus_{\pi\in Irr(\widehat{\G})} P_\pi\in \ell^\infty(\G)$ such that $I = \ell^\infty(\G)P$. Then
$$E = (P\fH_\pi)_{\pi\in Irr(\widehat{\G})} = (P_\pi\fH_\pi)_{\pi\in Irr(\widehat{\G})}$$
is a closed quantum subset of $\widehat{\G}$. So, we have a one-to-one correspondence between closed quantum subsets of $\widehat{\G}$, orthogonal projections in $\ell^\infty(\G)$, and weak$^*$ closed left ideals in $\ell^\infty(\G)$. They may also be detected as follows (see the analogous result for coideals of LCQGs \cite[Proposition 1.5]{K18}).
\begin{prop}\label{dualityInvSpaces}
    The following hold:
    \begin{enumerate}
        \item $P_E$ is the orthogonal projection onto $L^2(R_{\widehat{\G}}(L^\infty(\widehat{E})))$;
        \item $\ell^\infty(\G)(1-P_E) = \overline{\lambda_{\widehat{\G}}(j(E))}^{wk*}$;
        \item $\overline{\lambda_\G(\ell^1(\G)P_E)}^{wk*} = X^*(E)$ where $X^*(E) = \{R_{\widehat{\G}}(x)^* : x\in L^\infty(\widehat{E})\}$;
        \item $\overline{\lambda_\G(P_E\ell^1(\G))}^{wk*} = R_{\widehat{\G}}(L^\infty(\widehat{E}))$.
    \end{enumerate}
\end{prop}
\begin{proof}
    1. We refer the reader to \cite[Section 2.1]{Wang16}. For each $\pi\in Irr(\widehat{\G})$ fix an ONB $\{e_i^\pi\}$ so that $F_\pi$ is diagonal. It was established with \cite[Proposition 2.1.2]{Wang16} that the Fourier transform
    $$\fF : L^2(\widehat{\G})\to \ell^2(\G),~ \eta_{\widehat{\G}}(x)\mapsto \eta_\G(\lambda_{\widehat{\G}}(\hat{x})),$$
    where $\widehat{x}(y) = h(x^*y)$, is a unitary operator, and furthermore, the elements $E_{i,j}^\pi$ are identified with the elements $tr(F_\pi)(F_\pi)_{i,i}^{-1}u_{i,j}^\pi$. With \cite[Proposition 2.1.2]{Wang16}, one obtains the decomposition $\ell^2(\G)\cong \ell^2-\bigoplus_{\pi\in Irr(\widehat{\G})} L^2(M_{n_\pi})$.
    
    Recall that
    $$M_{n_\pi}\cong C_{n_\pi}\otimes_{min}R_{n_\pi}, ~ E_{i,j}^\pi\mapsto e_i^\pi\otimes e_j^\pi$$
    where $C_{n_\pi}$ and $R_{n_\pi}$ are the column and row Hilbert spaces on $\fH_\pi$, and each $R_{n_\pi}$ is invariant with respect to the left regular representation of $\widehat{\G}$. The latter implies $P_E\ell^2(\G) = \oplus_{\pi\in Irr(\widehat{\G})}P_E\fH_\pi$. So, if we let $x_\pi = \sum_{i,j}^{n_\pi}c_{i,j}^\pi E^\pi_{i,j}\in M_{n_\pi}$, then $\eta_\G(x_\pi)\in P_E \ell^2(\G)$ if and only if $\xi = \sum_{j=1}^{n_\pi} c_{i,j}^\pi e_j^\pi\in P_E\fH_\pi$, which equivalently says $\xi\in P_E\ell^2(\G)$ if and only if $u^\pi_{\xi,\eta}\in \Pol(\widehat{E})$ for arbitrary $\eta\in \fH_\pi$. We deduce that $P_E\ell^2(\G) = L^2(L^\infty(\widehat{E}))$ using the above identification between $E_{i,j}^\pi$ and $tr(F_\pi)(F_\pi)_{i,i}^{-1}u_{i,j}^\pi$.
    
    2. We established in \cite{AS22} that for any $\pi\in Irr(\widehat{\G})$
    $$\pi(j(E)) = \{A\in M_{n_\pi} : A(P_E\fH_\pi) = 0\} = M_{n_\pi}(1-P_E).$$
    Then
    $$\lambda_{\widehat{\G}}(j(E)) = c_{00}(\G)(1- P_E),$$
    and the rest follows from weak$^*$ density of $c_{00}(\G)$ in $\ell^\infty(\G)$.
    
    3. For each $\pi\in Irr(\widehat{\G})$ choose an ONB so that $P_E$ is diagonal. So,
    $$\delta_{i,j}^\pi P_E = \begin{cases} \delta_{i,j}^\pi &\text{if} ~ E_{i,j}\in P_E\ell^\infty(\G) \\ 0 &\text{otherwise}\end{cases} = \begin{cases} \delta_{i,j}^\pi &\text{if} ~ e_i^\pi\in P_E\fH_\pi \\ 0&\text{otherwise}\end{cases}$$
    and $\lambda_\G(\delta_{i,j}^
    \pi) = (u^\pi_{j,i})^*$ So, $\lambda_\G(\ell^1_F(\G)P_E) = Pol(\widehat{E}))^*$ and the rest is clear.
    
    4. This follows from a similar argument to 3. and by using \eqref{unitaryAntipode} (cf. Section $4.1$).
\end{proof}
\begin{rem}\label{QuantumSetsRemark}
    First note that if $N$ is a right coideal, then $N = L^\infty(\widehat{E})$. Indeed, it follows from the work in \cite{K18} that the orthogonal projection $P$ onto $L^2(R_{\widehat{\G}}(N))$ is the associated group-like projection for $R_{\widehat{\G}}(N)$. Since $P\in \ell^\infty(\G)$, it must be the case  that $P = P_E$, and then from $3.$ of Proposition \ref{dualityInvSpaces} and \cite[Proposition 1.5]{K18} we deduce that $R_{\widehat{\G}}(N) = R_{\widehat{\G}}(L^\infty(\widehat{E}))$ (note that in \cite{K18} the right regular representation is used but we are using the left regular representation, and hence the corresponding results in \cite{K18} are on right coideals whereas ours are on left coideals).

    If $L^\infty(\widehat{E})$ is a right coideal and $\widetilde{N}_P\subseteq\ell^\infty(\G)$ is the codual, Proposition \ref{dualityInvSpaces} $1.$ tells us $P_E = R_\G(P)$. Indeed, the orthogonal projection onto $L^2(R_{\widehat{\G}}(L^\infty(\widehat{E}))$ is a group-like projection that generates the {\it left} coideal $R_\G(\widetilde{N}_P)$ (see \cite{K18}).
\end{rem}
This means we should be able to glean information from $L^\infty(\widehat{E})$ using the projection $P_E$. For instance, the right coideals are in $1$-$1$ one correspondence with the group-like projections.
\begin{prop}\cite[Proposition 1.5]{K18}
    We have that $L^\infty(\widehat{E})$ is a right coideal if and only if $P_E$ is group-like.
\end{prop}
Our next result is concerned about two-sidedness of invariant subspaces. It is something that is well-known for coideals.
\begin{defn}
    Let $\G$ and $\H$ be DQGs. We say $\widehat{\H}$ is a {\bf closed quantum subgroup} of $\widehat{\G}$ if there exists a surjective unital $*$-homomorphism $\pi_\H : \Pol(\widehat{\G}) \to \Pol(\widehat{\H})$ satisfying
    $$(\pi_\H\otimes\pi_\H)\circ\Delta_\G = \Delta_\H\circ\pi_\H.$$
    Equivalently, there is an analogous $*$-homomorphism $\pi^u_\H : C_u(\widehat{\G})\to C_u(\widehat{\H})$.
\end{defn}
Given DQGs $\G$ and $\H$ where $\widehat{\H}$ is a closed quantum subgroup of $\widehat{\G}$, we define the quotient space:
$$\Pol(\widehat{\G}/\widehat{\H})= \{a\in \Pol(\widehat{\G}) : (\id\otimes\pi_\H)\circ\Delta_\G(a) = a\otimes 1\}.$$
Then, we set $L^\infty(\widehat{\G}/\widehat{\H}) = \overline{\Pol(\widehat{\G}/\widehat{\H})}^{wk*}$ etc. The von Neumann algbera $L^\infty(\widehat{\G}/\widehat{\H})$ is a right coideal of $L^\infty(\widehat{\G})$. Similarly to the discrete case, we will use the notation $J^1(\widehat{\G},\widehat{\H}) = L^\infty(\widehat{\G}/\widehat{\H})_\perp$.

We point out a quantum analogue of the Herz restriction theorem (cf. \cite{Daws}). Let $\G$ and $\H$ be a LCQGs. We say $\H$ is a (Vaes) closed quantum subgroup of $\G$ if there exists a normal unital injective $*$-homomorphism $\gamma_\H : L^\infty(\widehat{\H})\to L^\infty(\widehat{\G})$ such that
$$(\gamma_\H\otimes\gamma_\H)\circ\Delta_{\widehat{\H}} = \Delta_{\widehat{\G}}\circ\gamma_\H.$$
It turns out that the closed quantum subgroups of DQGs and CQGs are (Vaes) closed quantum subgroups. It is clear from the definition that if $\H$ is a (Vaes) closed quantum subgroup the LCQG $\G$, then $L^\infty(\widehat{\H})$ is both a two-sided coideal of $L^\infty(\widehat{\G})$. In fact, $L^\infty(\widehat{\H})$ is the codual of $L^\infty(\G/\H)$.

Using the duality between $L^1(\G)$-submodules of $L^\infty(\G)$ and $L^1(\G)$, we see that $X$ is two-sided if and only if $J^1(X)$ is a two-sided ideal. If $\G$ is discrete and $E$ is a closed quantum subset and $L^\infty(\widehat{E})$ is a right coideal, it is clear that $L^\infty(\widehat{E})$ is two-sided if and only if $\Delta_{\widehat{\G}}(L^\infty(\widehat{E}))\subseteq L^\infty(\widehat{E})\overline{\otimes}L^\infty(\widehat{E})$. The following is probably well-known for coideals.
\begin{prop}\label{centralProjs}
    Let $\G$ be a DQG and $E$ a closed quantum subset. We have that $L^\infty(\widehat{E})$ is two-sided if and only if $P_E$ is central.
\end{prop}
\begin{proof}
    If $P_E$ is central, then we have $P_E\fH_\pi = \fH_\pi$ or $\{0\}$. Consequently, it follows by definition of $\Pol(\widehat{E})$ that if $u_{i,j}^\pi\in \Pol(\widehat{E})$ for any $i,j$, then $u_{i,j}^\pi\in \Pol(\widehat{E})$ for every $i,j$. It is then clear that $\Delta_{\widehat{\G}}(u_{i,j}^\pi)\in \Pol(\widehat{E})\otimes \Pol(\widehat{E})$. So $\Delta_{\widehat{\G}}(a)\in \Pol(\widehat{E})\otimes \Pol(\widehat{E})$ for every $a\in \Pol(\widehat{E})$. By weak$^*$ density, we conclude that $\Delta_{\widehat{\G}}(L^\infty(\widehat{E}))\subseteq L^\infty(\widehat{E})\overline{\otimes}L^\infty(\widehat{E})$.
    
    Conversely, if $L^\infty(\widehat{E})$ is two-sided, using linear independence of the sets $\{u^\pi_{i,j_0} : 1\leq i\leq n_\pi\}$ and $\{u_{i_0,j}^\pi : 1\leq j\leq n_\pi\}$ for fixed $i_0$ and $j_0$, and the fact
    $$\Delta_{\widehat{\G}}(u_{i,j}^\pi) = \sum_{t=1}^{n_\pi}u_{i,t}^\pi\otimes u_{t,j}^\pi\in \Pol(\widehat{E})\otimes \Pol(\widehat{E})$$
    it follows that if $u^\pi_{i,j}\in \Pol(\widehat{E})$, then $u^\pi_{i,j}\in \Pol(\widehat{E})$ for every $i,j$. Consider $P = \oplus_{\pi\in Irr(\widehat{\G})} P_\pi$ where $P_\pi = I_{n_\pi}$ if $u_{i,j}^\pi\in \Pol(\widehat{E})$ and $0$ otherwise. Then
    $$\Pol(\widehat{E}) = \{u_{\xi,\eta}^\pi : \xi,\eta\in P\fH_\pi, \pi\in Irr(\widehat{\G})\}.$$
\end{proof}

\subsection{Coamenable Compact Quasi-Subgroups}
From now on, we will focus our attention on the compact quasi-subgroups of $L^\infty(\widehat{\G})$ for a DQG $\G$. Our goal is establish some basic properties of coamenable coideals. Given a functional $\mu\in M^u(\widehat{\G})$, we let $R_\mu$ and $L_\mu$ be the adjoints of the maps
$$\nu\mapsto \nu*\mu ~ \text{and} ~ \nu\mapsto \mu*\nu, ~ \nu\in L^1(\widehat{\G})$$
respectively.
\begin{defn}
    A {\bf compact quasi-subgroup} of a CQG $\widehat{\G}$ is a right coideal of the form $R_\omega(L^\infty(\widehat{\G}))$ for an idempotent state $\omega\in M^u(\widehat{\G})$. We denote $N_\omega = R_\omega(L^\infty(\widehat{\G}))$.
\end{defn}
Recall that the fundamental unitary $W_\G\in L^\infty(\G)\overline{\otimes} L^\infty(\widehat{\G})$ of a LCQG $\G$ admits a `half-lifted' version $W_\G^u \in M(C_u(\G)\otimes_{min} C_r(\widehat{\G}))$ such that $(\Gamma_\G\otimes \id)(W_\G^u) = W_\G$ (see \cite{K01}). Then $\lambda_\G$ extends to a representation $\lambda_\G^u : M^u(\widehat{\G})\to L^\infty(\widehat{\G})$ where $\lambda_\G^u(\mu) = (\mu\otimes\id)W_\G^u$ and $\lambda^u_\G|_{L^1(\G)} = \lambda_\G$. We will abuse notation and write $\lambda_\G$ for the extended version. Note that when $\G$ is discrete, every $\pi\in Irr(\widehat{\G})$ extends to $M^u(\widehat{\G})$ and satisfies $\pi(\mu) = [\mu(u_{i,j}^\pi)]_{i,j}$ and $\lambda_{\widehat{\G}}(\mu) = \oplus_{\pi\in Irr(\widehat{\G})}\pi(\mu)$.

Let $\G$ be a DQG. The projection $P = \lambda_{\widehat{\G}}(\omega)$ is group-like, and we say the right coideals $N_\omega$ and $\widetilde{N}_P$ are {\bf codual} coideals. We sometimes use the notation $\widetilde{N_\omega} = \widetilde{N}_P$ and $\widetilde{\widetilde{N}_P} = N_\omega$. In the particular case where $\widehat{\H}$ is a closed quantum subgroup of $\widehat{\G}$, $\ell^\infty(\widehat{\H}) = \widetilde{L^\infty(\G/\H)}$.
\begin{rem}
    Our terminology is not faithful to the literature. The codual of a right coideal $N$ of a LCQG $\G$ is typically defined to be the {\it left} coideal $N'\cap L^\infty(\widehat{\G})$. For discrete $\G$, it turns out that
    $$\widetilde{N}_P = R_\G(N_\omega'\cap \ell^\infty(\G)) ~ \text{\cite[Lemma 2.6]{T07}}.$$
\end{rem}
Using the formulas for $\omega$ on the matrix coefficients \cite[Lemma 4.7]{AS22} and the decomposition of the left regular representation into irreducibles, it is straightforward checking that we have
$$(R_\omega\otimes \id)W_{\widehat{\G}} = W_{\widehat{\G}}(1\otimes P).$$
See \cite{Soltan} for an account of compact quasi-subgroups at the level of LCQGs.

Given an idempotent state $\omega\in M^u(\widehat{\G})$, we let
$$R^u_\omega = (\id\otimes \omega)\circ\Delta^u_{\widehat{\G}} : C_u(\widehat{\G})\to C_u(\widehat{\G})$$
denote the universal version, and
$$R^r_\omega = R_\omega|_{C_r(\widehat{\G})} : C_r(\widehat{\G})\to C_r(\widehat{\G})$$
denote the reduced version. Likewise with $L^u_\omega$ and $L^r_\omega$. It turns out that $\Gamma_{\widehat{\G}}\circ R^u_\omega = R^r_\omega \circ \Gamma_{\widehat{\G}}$.

We first recount what was established in \cite{AS22} (see also \cite[Section 2]{FLS16}). Let $N_\omega$ be a compact quasi-subgroup of $L^\infty(\widehat{\G})$, so $N_\omega = R_\omega(L^\infty(\widehat{\G})) = L^\infty(\widehat{E_\omega})$ for some idempotent state $\omega\in M^u(\widehat{\G})$ and hull $E_\omega$. We have that $P_{E_\omega} = \lambda_{\widehat{\G}}(\omega)$. Then
$$\Pol(\widehat{E}_\omega) = R_\omega(\Pol(\widehat{\G})).$$
Now, let
$$C_u(\widehat{E}_\omega) = \overline{\Pol(\widehat{E}_\omega)}^{||\cdot||_u}\subseteq C_u(\widehat{\G}),$$
and
$$C_r(\widehat{E}_\omega) = \Gamma_{\widehat{\G}}(C_u(\widehat{E}_\omega)).$$
Note, then, that it follows that
$$C_r(\widehat{E}_\omega) = R_\omega^r(C_r(\widehat{\G})) ~ \text{and} ~ C_u(\widehat{E}_\omega) = R^u_\omega(C_u(\widehat{\G})).$$
We will also set $M^r(\widehat{E}_\omega) = C_r(\widehat{E}_\omega)^*$ and $M^u(\widehat{E}_\omega) = C_u(\widehat{E}_\omega)^*$.
\begin{defn}\label{coamenQuot}
    Let $\G$ be a DQG and $N_\omega$ a compact quasi-subgroup of $\widehat{\G}$. We say $N_\omega$ is {\bf coamenable} if there exists a state $\epsilon_{N}\in M^r(\widehat{E}_\omega)$ such that
    $$\epsilon_{N}\circ \Gamma_{\widehat{\G}}|_{C_u(\widehat{E}_\omega)} = \epsilon^u_{\widehat{\G}}|_{C_r(\widehat{E}_\omega)}.$$
\end{defn}
Note that $C_r(\widehat{E_\omega})*\mu\subseteq C_r(\widehat{E_\omega})$ for all $\mu\in M^r(\widehat{\G})$. Suppose $N_\omega$ is coamenable, with associated state $\epsilon_N$. A consequence of coameability of $N_\omega$ is that
$$(\mu\otimes\epsilon_N)(\Delta_{\widehat{\G}}(a)) = \mu(a)$$
for all $a\in C_r(\widehat{E_\omega})$ and $\mu\in M^r(\widehat{\G})$, or,
$$\mu*(\epsilon_N\circ R_\omega^r) = (\mu*\epsilon_N)\circ R_\omega^r = u\circ R_\omega^r$$
for all $\mu\in M^r(\widehat{\G})$.
\begin{prop}\label{coamen1}
    $N_\omega$ is coamenable if and only if there exists a state $\epsilon_N^r\in M^r(\widehat{\G})$ such that $\epsilon_N^r\circ\Gamma_{\widehat{\G}}|_{C_u(\widehat{E}_\omega)} = \epsilon^u_{\widehat{\G}}|_{C_u(\widehat{E}_\omega)}$.
\end{prop}
\begin{proof}
    Suppose $N_\omega$ is coamenable with state $\epsilon_N\in M^r(\widehat{E}_\omega)$ as in the definition. Then $\epsilon_N\circ R_\omega^r\in M^r(\widehat{\G})$ is a state, and we have
    \begin{align*}
        \epsilon_N\circ R_\omega^r\circ \Gamma_{\widehat{\G}}|_{C_u(\widehat{E}_\omega)} &= \epsilon_N\circ\Gamma_{\widehat{\G}} \circ R_\omega^u|_{C_u(\widehat{E}_\omega)} = \epsilon_{\widehat{\G}}^u |_{C_u(\widehat{E}_\omega)}.
    \end{align*}
    Conversely, if $\epsilon^r_N\in M^r(\widehat{\G})$ is a state such as in the hypothesis, then it is straightforward to show that $\epsilon_N^r|_{C_r(\widehat{E}_\omega)}\in M^r(\widehat{E}_\omega)$ is a state that makes $N_\omega$ coamenable.
\end{proof}
As we are about to see, the counit associated with $N_\omega$ is actually $\omega$. Thus $N_\omega$ is coamenable if and only if $\omega\in M^r(\widehat{\G})$.
\begin{thm}\label{reducedIdempState}
    Let $\G$ be a DQG and $N_\omega$ a compact quasi-subgroup of $\widehat{\G}$. We have that $N_\omega$ is coamenable if and only if $\omega\in M^r(\widehat{\G}) = C_r(\widehat{\G})^*$.
\end{thm}
\begin{proof}
    Suppose $\omega\in M^r(\widehat{\G})$. Using \cite[Lemma 4.7]{AS22}, for each $\pi\in Irr(\widehat{\G})$, choose an ONB $\{e_j^\pi\}$ that diagonalizes $\pi(\omega)$. Then,
    $$\Pol(\widehat{E}_\omega) = \Span\{u_{i,j}^\pi : 1\leq i\leq n_\pi, \pi(\omega)e_j^\pi = e_j^\pi\}.$$
    For $u_{i,j}^\pi\in \Pol(\widehat{E}_\omega)$, $\omega(u_{i,j}^\pi) = \delta_{i,j} = \epsilon^u_{\widehat{\G}}(u_{i,j}^\pi)$. By density $\omega\circ\Gamma_{\widehat{\G}}|_{C_u(\widehat{E}_\omega)} = \epsilon^u_{\widehat{\G}}|_{C_u(\widehat{E}_\omega)}$.
    
    Conversely, from Proposition \ref{coamen1}, there exists a state $\epsilon_N^r\in M^r(\widehat{\G})$ such that $\epsilon_N^r\circ\Gamma_{\widehat{\G}}|_{C_u(\widehat{E}_\omega)} = \epsilon^u_{\widehat{\G}}|_{C_u(\widehat{E}_\omega)}$. In the proof of Proposition \ref{coamen1} we see that it can be arranged that there exists $\epsilon_N\in M^r(\widehat{E}_\omega)$ such that $\epsilon_N\circ R_\omega^r = \epsilon^r_N$. In particular, we may arrange the property $\epsilon^r_N\circ R_\omega^r = \epsilon^r_N$. Thence,
    \begin{align*}
        R_\omega^r&= (\id\otimes\epsilon^r_N)\circ\Delta^r_{\widehat{\G}}\circ R_\omega^r = (\id\otimes(\epsilon^r_N\circ R_\omega^r))\circ\Delta^r_{\widehat{\G}} = R_{\epsilon^r_N}^r.
    \end{align*}
    By injectivity of the map $\mu\mapsto R^r_\mu$ we deduce that $\omega = \epsilon^r_N\in M^r(\widehat{\G})$.
\end{proof}
Recall that a compact quantum group is of {\bf Kac} type if its Haar state is tracial. It is well-known that a tracial idempotent state on $C_u(\widehat{\G})$ is automatically the Haar state on some Kac closed quantum subgroup $\widehat{\H}$ of $\widehat{\G}$. Indeed, if $\omega$ is tracial, then
$$\{a\in C_u(\widehat{\G}) : \omega(a^*a) = 0\}$$
is two-sided, and hence $\omega = h_{\widehat{\H}}\circ\pi_{\widehat{\H}}^u$, where $\pi_{\widehat{\H}}^u : C_u(\widehat{\G})\to C_u(\widehat{\H})$ is the morphism implementing $\widehat{\H}$ as a closed quantum subgroup of $\widehat{\G}$ (see \cite[Theorem 5]{SS16}). Moreover, $h_{\widehat{\H}}$ is tracial if and only if $\omega$ is. This, combined with Theorem \ref{reducedIdempState} gives us the following.
\begin{cor}
    Let $\G$ be a DQG. The tracial idempotent states in $M^r(\widehat{\G})$ are in one-to-one correspondence with the closed quantum subgroups of $\widehat{\G}$ of Kac type with coamenable quotient.
\end{cor}
Recall that $P = \lambda_{\widehat{\G}}(\omega) = (\omega\otimes\id)W_{\widehat{\G}}^u$ is the group-like projection generating the codual of $N_\omega$. Coamenability of $N_\omega$ means that we may weak$^*$ approximate $\omega\in M^r(\widehat{\G})$ with states $(e_j)\subseteq L^1(\widehat{\G})$. These states satisfy the property
$$f(e_j\otimes \id)W_{\widehat{\G}}\to f(\omega\otimes \id)W_{\widehat{\G}}^u = f(P), f\in \ell^1(\G).$$
With these observations, we can establish coamenability of a compact quasi-subgroup in terms of almost invariant vectors in $\ell^2(\G)$.
\begin{cor}\label{invVecs}
    If $N_\omega$ is coamenable then there exists a net of unit vectors $(\xi_j)\subseteq P\ell^2(\G)$ such that for $\eta\in\ell^2(\G)$,
    $$||W_{\widehat{\G}}(\xi_j\otimes P\eta) - \xi_j\otimes P\eta||_2 \to 0.$$
\end{cor}
\begin{proof}
    From Theorem \ref{reducedIdempState}, we have that $\omega\in M^r(\widehat{\G})$. Let $(w_j)\subseteq L^1(\widehat{\G})$ be a net of states weak$^*$ approximating $\omega$. For this proof, we will be forced to consider the left coideal $L_\omega(L^\infty(\widehat{\G})) = R_{\widehat{\G}}(N_\omega)$ By idempotency of $\omega$, $(w_j\circ L_\omega)\subseteq L^1(\widehat{\G})$ is still a net of states that weak$^*$ approximates $\omega$. Since $L_\omega\circ L_\omega = L_\omega$, we may assume $w_j\circ L_\omega = w_j$.
    
    The restriction $w_j|_{R_{\widehat{\G}}(N_\omega)}\in (R_{\widehat{\G}}(N_\omega))_*$ is a state, so, we can find a unit vector $\xi_j\in L^2(R_{\widehat{\G}}(N_\omega))$ such that $w_j|_{R_{\widehat{\G}}(N_\omega)} = w_{\xi_j}|_{R_{\widehat{\G}}(N_\omega)}$. We want to show $w_j = w_{\xi_j}|_{L^\infty(\widehat{\G})}$. For $x\in L^\infty(\widehat{\G})$,
    $$w_j(x) = w_{\xi_j}(L_\omega(x)) = \langle L_\omega(x)\xi_j,\xi_j\rangle.$$
    Where $\eta_{\widehat{\G}}$ is the GNS map of the left Haar weight, using the equation,
    $$P\eta_{\widehat{\G}}(x) = \eta_{\widehat{\G}}(L_\omega(x)) ~ \text{(cf. left version of work in \cite{Soltan})},$$
    for $y\in N_\omega$ and $\zeta\in L^2(R_{\widehat{\G}}(N_\omega))$ we get
    \begin{align*}
        w_{\eta_{\widehat{\G}}(y), \zeta}(L_\omega(x)) &= \langle L_\omega(x)\eta_{\widehat{\G}}(y),\zeta\rangle = \langle \eta_{\widehat{\G}}(L_\omega(x)y), \zeta\rangle
        \\
        &= \langle \eta_{\widehat{\G}}(L_\omega(xy)), \zeta\rangle = \langle Px\eta_{\widehat{\G}}(y), \zeta\rangle = \langle x\eta_{\widehat{\G}}(y), \zeta\rangle
    \end{align*}
    where we used the fact $L_\omega$ is a $R_{\widehat{\G}}(N_\omega)$-bimodule map and that $L^2(R_{\widehat{\G}}(N_\omega)) = P\ell^2(\G)$. Using density of $\eta_{\widehat{\G}}(R_{\widehat{\G}}(N_\omega))$ in $L^2(R_{\widehat{\G}}(N_\omega))$ we get
    $$w_j(x) = w_{\xi_j}(x).$$
    The rest of the proof is an adaptation of the case where $\omega = \epsilon^u_{\widehat{\G}}$ (cf. \cite[Theorem 3.12]{B16}). For $\eta\in\ell^2(\G)$,
    \begin{align*}
        ||W_{\widehat{\G}}(\xi_j\otimes P\eta) - \xi_j\otimes P\eta||_2 &= 2||P\eta||^2 - 2Re\langle W_{\widehat{\G}}(\xi_j\otimes P\eta), \xi_j\otimes P\eta\rangle
        \\
        &= 2||P\eta||^2 - 2Re (w_{\xi_j}\otimes w_{P\eta})(W_{\widehat{\G}})
        \\
        &\to 2||P\eta||^2 - 2w_{P\eta}(P) = 0.
    \end{align*}
\end{proof}
As an application of our work on coamenable compact quasi-subgroups, we find a characterization of the central idempotent states on $C_r(\widehat{\G})$.

We say a quantum subgroup $\H$ of $\G$ is {\bf normal} if $\ell^\infty(\G/\H)$ is a two-sided coideal.
\begin{cor}\label{redcenidemps}
    Let $\G$ be a DQG. There is a one-to-one correspondence between the amenable quantum subgroups of $\G$ and the central idempotent states on $C_r(\widehat{\G})$. The tracial central idempotent states on $C_r(\widehat{\G})$ are in one-to-one correspondence with the amenable normal quantum subgroups of $\G$ for which their quotients are unimodular.
\end{cor}
\begin{proof}
    It was shown with \cite[Theorem 4.3]{Kal1} that there is a one-to-one correspondence between central group-like projections in $\ell^\infty(\G)$ and quantum subgroups of $\G$. Then, the extension of $\lambda_{\widehat{\G}}$ to $M^u(\widehat{\G})$ gives us a one-to-one correspondence between central group-like projections in $\ell^\infty(\G)$ and idempotent states on $C_u(\widehat{\G})$ \cite[Theorem 4.3]{FK17}. Let $\H$ be a quantum subgroup of $\G$ and $1_\H$ the central group-like projection that generates $\ell^\infty(\G/\H)$. Let $1_\H = \lambda_{\widehat{\G}}(\omega)$ for some central idempotent state $\omega : C_u(\widehat{\G})\to\C$.
    
    It is readily seen that the definition of coamenability of $N_\omega = L^\infty(\widehat{\H})$ is equivalent to coamenability of $\widehat{\H}$. Indeed, let $E_\H$ be the quantum subset for the coideal $L^\infty(\widehat{\H}) = L^\infty(\widehat{E_\H})$. Then
    $$C_r(\widehat{E_\H}) = R_\omega^r(C_r(\widehat{\G})) = C_r(\widehat{\H})$$
    and similarly $\Pol(\widehat{\H}) = \Pol(\widehat{E_\H})$. Also, $\epsilon_{\widehat{\G}}^u|_{\mathrm{Pol}(\widehat{\H})} = \epsilon_{\widehat{\H}}$ and so $L^\infty(\widehat{\H})$ is coamenable if and only if $\epsilon_{\widehat{\H}}$ extends continuously to $C_r(\widehat{\H})$. Then, using Theorem \ref{reducedIdempState}, we find that $\widehat{\H}$ is coamenable if and only if $\omega\in M^r(\widehat{\G})$.
    
    Recall that a discrete quantum group $\G$ is unimodular if and only if the Haar state of $\widehat{\G}$ is tracial. The second claim follows from the duality between normal quantum subgroups of $\G$ and normal quantum subgroups of $\widehat{\G}$ (see\cite{Daws}). Indeed, $\H$ is normal if and only if $\widehat{\G/\H}$ is a closed quantum subgroup of $\widehat{\G}$.
    
    Suppose $1_\H = \lambda_{\widehat{\G}}(\omega_{\widehat{\G/\H}})$ where $\omega_{\widehat{\G/\H}} = h_{\widehat{\G/\H}}\circ\pi_{\widehat{\G/\H}}$ is the Haar state on $\widehat{\G/\H}$. It is easy to see that $\omega_{\widehat{\G/\H}}$ is tracial whenever $h_{\widehat{\G/\H}}$ is tracial.
    
    Conversely, if $1_\H = \lambda_{\widehat{\G}}(\omega)$ and $\omega$ is tracial, then $\omega$ must be of Haar type because then
    $$\{a\in C_u(\widehat{\G}) : \omega(a^*a)= 0\}$$
    is a two-sided ideal (see \cite[Theorem 5]{SS16}). It follows that $\widehat{\H}$ is a quotient of $\widehat{\G}$, and hence $\omega = \omega_{\widehat{\G/\H}}$ is the tracial Haar state on $\widehat{\G/\H}$ (see \cite{Daws} for more).
\end{proof}
\begin{rem}
    It is clear that if $\omega\in M^r(\widehat{\G})$ then $\epsilon^u_{\widehat{\H}}|_{\mathrm{Pol}(\widehat{\H})}$ extends to $C_r(\widehat{\H})$ since $\omega|_{\mathrm{Pol}(\widehat{\H})} = \epsilon^u_{\widehat{\H}}|_{\mathrm{Pol}(\widehat{\H})}$. The converse, however, is unclear without Theorem \ref{reducedIdempState}. If $\widehat{\H}$ is coamenable, then $\epsilon_{\widehat{\H}} \in M^r(\widehat{\H})$ and $\omega|_{\mathrm{Pol}(\widehat{\H})} = \epsilon_{\widehat{\H}}|_{\mathrm{Pol}(\widehat{\H})}$, but we do not immediately have that $\omega|_{\mathrm{Pol}(\widehat{\G})}$ extends to $M^r(\widehat{\G})$.
\end{rem}
\subsection{Duality of Amenability and Coamenability}
We will fix an idempotent state $\omega\in M^u(\widehat{\G})$ (and hence a compact quasi-subgroup $N_\omega\subseteq L^\infty(\widehat{\G})$ and its hull $E_\omega$). We set $P = \lambda_{\widehat{\G}}(\omega)$. We require a certain lemma before proceeding.

Set $B_\omega = R_\omega\circ L_\omega$, which, from coassociativity, is a ucp projection (but possibly without $L^1(\widehat{\G})$-module properties). We will denote the subspace
$$\Pol(\widehat{E}_\omega)\cap R_{\widehat{\G}}(\Pol(\widehat{E}_\omega)) = {\Pol}_B(\widehat{E}_\omega) = \{u_{\xi,\eta}^\pi : \xi,\eta\in P\fH_\pi, ~\pi\in Irr(\widehat{\G})\} = B_\omega(\Pol(\widehat{\G})).$$
So,
$$\overline{{\Pol}_B(\widehat{E}_\omega)}^{wk*} = N_\omega\cap R_{\widehat{\G}}(N_\omega) = B_\omega(L^\infty(\widehat{\G})).$$
Set $C^u_B(\widehat{E}_\omega) = \overline{\Pol_B(\widehat{E}_\omega)}^{||\cdot||_u}$, $C^r_B(\widehat{E}_\omega) = \Gamma_{\widehat{\G}}(C^u_B(\widehat{E}_\omega))$, and $M_B^r(\widehat{E_\omega}) = C^r_B(\widehat{E}_\omega)^*$. A similar proof to Proposition \ref{dualityInvSpaces} will show $\overline{\lambda_\G(P\ell^1(\G)P)}^{wk*} = B_\omega(L^\infty(\widehat{\G}))$.
\begin{lem}\label{MainThmLem}
    If there exists a net of unit vectors $(\xi_j)\subseteq\ell^2(\G)$ such that
    \begin{align}\label{amenveccond}
        ||\lambda_\G(PfP)\xi_j - f(P)\xi_j||_2\to 0, ~ f\in\ell^1(\G)
    \end{align}
    then $N_\omega$ is coamenable.
\end{lem}
\begin{proof}
    The proof follows from a similar statement in the proof that amenability of $\G$ implies coamenability of $\widehat{\G}$ (cf. \cite[Theorem 3.15]{B16}). Consider
    $$\epsilon_P : P\ell^1(\G)P\to \C, ~ f\mapsto f(1).$$
    Then \eqref{amenveccond} tells us $|f(P)| \leq ||\lambda_\G(PfP)||$ for every $f\in \ell^1(\G)$, so $\epsilon_{\widetilde{N}_P}$ extends to a functional $\widetilde{\epsilon_P}\in M^r_B(\widehat{E}_\omega)$. For each $\pi\in Irr(\widehat{\G})$, using \cite[Lemma 4.7]{AS22}, we can choose an ONB $\{e_j^\pi\}$ of $\fH_\pi$ that diagonalizes $\pi(\omega)$. Then,
    $$\Pol(\widehat{E}_\omega) = \Span\{u_{i,j}^\pi : 1\leq i\leq n_\pi, \pi(\omega)e_j^\pi = e_j^\pi\}.$$
    Then, since $\omega\circ R_{\widehat{\G}} = \omega$ \cite[Proposition 4]{SS16}, we have $\omega\circ R_{\widehat{\G}}(u_{i,j}^\pi) = \delta_{i,j}$ if $u_{i,j}^\pi\in R_{\widehat{\G}}(\Pol(\widehat{E}_\omega))$ and zero otherwise. So, for $u^\pi_{i,j}\in \Pol(\widehat{\G})$,
    \begin{align*}
        \widetilde{\epsilon_P}\circ B_\omega(u^\pi_{i,j}) &= (\omega\otimes \widetilde{\epsilon_P}\otimes\omega)(\sum_{t,s=1}^{n_\pi}u^\pi_{i,t}\otimes u^\pi_{t,s}\otimes u^\pi_{s,j})
        \\
        &= \omega(u^\pi_{i,i})\omega(u^\pi_{j,j})\widetilde{\epsilon_P}(u_{i,j}^\pi)
        \\
        &= \omega(u^\pi_{i,i})\omega(u^\pi_{j,j})\delta_{i,j}^\pi(P)
        \\
        &= \omega(u_{i,j}^\pi).
    \end{align*}
    By density of $\Pol(\widehat{\G})$ in $C_r(\widehat{\G})$, we deduce that $\omega = \widetilde{\epsilon_P}\circ B_\omega\in M^r(\widehat{\G})$ and we apply Theorem \ref{reducedIdempState}.
\end{proof}
The following results illustrate how can we relative amenability and coamenability of coideals via Pontryagin duality. In both proofs we use an adaptation of the proof that $\G$ is amenable if and only if $\widehat{\G}$ is coamenable (due to \cite{Toma} but we follow \cite{B16}).
\begin{thm}\label{biglemmaconverse}
    Let $\G$ be a DQG and $N_\omega\subseteq L^\infty(\widehat{\G})$ a compact quasi-subgroup with $P =\lambda_{\widehat{\G}}(\omega)$. If $N_\omega$ is coamenable then $M_P$ is amenable.
\end{thm}
\begin{proof}
    Assume $N_\omega$ is coamenable. Using Lemma \ref{invVecs}, obtain a net of unit vectors $(\xi_j)\subseteq P\ell^2(\G)$ such that
    $$||W_{\widehat{\G}}(\xi_j\otimes P\eta) - \xi_j\otimes P\eta||_2\to 0.$$
    Since $W_{\widehat{\G}} = \Sigma W_\G^* \Sigma$, we have
    $$||W_\G(P\eta\otimes\xi_j) - (P\eta\otimes\xi_j)||_2\to 0.$$
    So, for $w_{\eta,\zeta} = f\in \ell^1(\G)$ and $x\in\ell^\infty(\G)$,
    \begin{align*}
        &|(Pw_{\eta,\zeta}P*w_{\xi_j}(x) - w_{\eta,\zeta}(P)w_{\xi_j}(x)|
        \\
        &\leq |\langle (1\otimes x)[W_\G (P\eta\otimes \xi_j) - (P\eta\otimes\xi_j)], W_\G(P\zeta\otimes \xi_j)\rangle|
        \\
        &+|\langle (1\otimes x) P\eta\otimes \xi_j, W_\G(P\zeta\otimes \xi_j) - P\zeta\otimes \xi_j\rangle|
        \\
        &= |\langle (1\otimes x)[W_\G (P\eta\otimes \xi_j) - (P\eta\otimes\xi_j)], W_\G(P\zeta\otimes \xi_j)\rangle|
        \\
        &+|\langle (1\otimes x) P\eta\otimes \xi_j, W_\G(P\zeta\otimes \xi_j) - P\zeta\otimes \xi_j\rangle|
        \\
        &\leq ||x||\,||P\zeta||\,||W_\G(P\eta\otimes\xi_j) - P\eta\otimes \xi_j||_2 + ||x||\,||P\eta||\,||W_\G(P\zeta\otimes \xi_j) - P\zeta\otimes \xi_j||_2
        \\
        &\to 0.
    \end{align*}
    If we let $m$ be a weak$^*$ cluster point of the net $(w_{\xi_j}|_{\ell^\infty(\G)})$, then it is straightforward to show $m$ is a state satisfying $(PfP)*m = f(P)m$ for all $f\in\ell^1(\G)$. Finally, since $(\xi_j)\subseteq P\ell^2(\G)$, $w_{\xi_j}(P) = 1$ for all $j$, so $m(P) = 1$. Using Proposition \ref{RelAmenIffAmenMP} and Theorem \ref{CondRelAmenMP}, we deduce that $M_P$ is amenable.
\end{proof}
It follows from the work of \cite{Kal1} that the central group-like projections in $\ell^\infty(\G)$ are in one-to-one correspondence with the quantum subgroups of $\G$, which are in one-to-one correspondence with the central idempotent states in $C_u(\widehat{\G})^*$. Here, we then have that $\widetilde{N}_P = \ell^\infty(\G/\H)$ where $\H$ is a quantum subgroup of $\G$ and $N_\omega = L^\infty(\widehat{\H})$ (see the proof of Corollary \ref{redcenidemps} for a full justification).
\begin{lem}\label{BigLemma}
    Let $\G$ be a DQG and $N_\omega\subseteq L^\infty(\widehat{\G})$ a compact quasi-subgroup such that $\omega$ is central. Denote $P =\lambda_{\widehat{\G}}(\omega)$. If $\widetilde{N}_P$ is relatively amenable then $N_\omega$ is coamenable.
\end{lem}
\begin{proof}
    The proof follows with very few changes to the proof in \cite{BV02}. We give a sketch for the benefit of the reader nonetheless. Let $m : \ell^\infty(\G)\to\C$ be a $P$-left invariant state. Using Lemma \ref{ApproxInv} and that $\ell^\infty(\G)$ is in standard form, we can find a net of unit vectors $(\xi_\alpha)\subseteq \ell^2(\G)$ such that $(w_{\xi_\alpha}|_{\ell^\infty(\G)})$ weak$^*$ approximates $m$, so that we have
    $$||Pf*w_{\xi_\alpha} - f(P)w_{\xi_\alpha}||_1\to 0, ~ f\in \ell^1(\G).$$
    
    Note that since $P$ is central, we either have $P_\pi = I_{n_\pi}$ or $P_\pi = 0$. When $P_\pi = 0$,
    $$||\lambda_\G(P_\pi f_\pi P_\pi)\xi_\alpha - f_\pi(P_\pi)\xi_\alpha||_2 = 0$$
    for every $f_\pi\in (M_{n_\pi})_*$
    
    Assume $P_\pi=  1_{n_\pi}$. Here the proof follows exactly as in \cite{BV02}, so we only give a sketch. Define the functionals $\mu_\alpha, \eta_\alpha\in (M_{n_\pi}(\ell^\infty(\G)))_*$ by setting
    \begin{align*}
        \eta_\alpha(x) = (\tr\otimes w_{\xi_\alpha})\left(x\right) = (\tr\otimes w_{\xi_\alpha})\left(x\right), ~ x = [x_{m,n}]\in M_{n_\pi}(\ell^\infty(\G))
    \end{align*}
    and
    \begin{align*}
        \mu_\alpha(x) = \sum_{n,m}^{n_\pi} \delta^\pi_{m,n}*w_{\xi_\alpha}(x_{m,n}), ~x = [x_{m,n}]\in M_{n_\pi}(\ell^\infty(\G)).
    \end{align*}
    As in \cite{BV02}, we find that $\eta_\alpha$ and $\mu_\alpha$ are positive functionals, and after some straightforward (but tedious) computations,
    \begin{align*}
        ||(F_\pi^{1/2} \otimes 1)\mu_\alpha(F_\pi^{1/2} \otimes 1) - (F_\pi^{1/2} \otimes 1)\eta_\alpha(F_\pi^{1/2} \otimes 1)||_{(M_{n_\pi}(\ell^\infty(\G)))_*}\to 0
    \end{align*}
    where $F_\pi$ is the $F$-matrix associated with $\pi$. It was shown in \cite{BV02} that
    \begin{align*}
        F_\pi^{1/2} \otimes\xi_\alpha\in L^2(M_{n_\pi})\otimes_2 \ell^2(\G)
    \end{align*}
    and
    \begin{align*}
        (F_\pi^{1/2}\otimes 1)[\lambda_\G(\delta_{k,l}^\pi)\xi_\alpha]_{l,k}\in L^2(M_{n_\pi})\otimes_2 \ell^2(\G)
    \end{align*}
    lie in the positive cone of $L^2(M_{M_{n_\pi}})\otimes_2 \ell^2(\G)$ and that
    \begin{align*}
        w_{F_\pi^{1/2} \otimes\xi_\alpha} = (F_\pi^{1/2} \otimes 1)\eta_\alpha(F_\pi^{1/2}\otimes 1)
    \end{align*}
    and
    \begin{align*}
        w_{(F_\pi^{1/2}\otimes 1)[\lambda_\G(\delta_{k,l}^\pi)\xi_\alpha]_{l,k}} = (F_\pi^{1/2} \otimes 1)\mu_\alpha(F_\pi^{1/2}\otimes 1).
    \end{align*}
    Using the Powers-St\o rmer inequality (cf. \cite{H75}), we have
    \begin{align*}
        &||(F_\pi^{1/2}\otimes 1)[\lambda(\delta_{k,l}^\pi)\xi_\alpha]_{l,k} - F_\pi^{1/2}\otimes\xi_\alpha ||_{L^2(M_{n_\pi})\otimes_2\ell^2(\G)} \to 0.
    \end{align*}
    Thus we deduce the following limit
    \begin{align*}
        &||[\lambda(\delta_{k,l}^\pi)\xi_\alpha]_{l,k} - 1_{n_\pi}\otimes\xi_\alpha ||_{L^2(M_{n_\pi})\otimes_2\ell^2(\G)} \to 0.
    \end{align*}
    Then, for $\sum \alpha_{i,j}\delta^\pi_{i,j} = f_\pi\in (M_{n_\pi})_*\subseteq\ell^1(\G)$
    \begin{align*}
        ||\lambda_\G(1_{n_\pi}f_\pi 1_{n_\pi})\xi_\alpha - f_\pi(1_{n_\pi})\xi_\alpha||_2 = ||(f_\pi^t\otimes\id)([\lambda_\G(\delta_{k,l}^\pi)\xi_\alpha]_{l,k} - 1_{n_\pi}\otimes\xi_\alpha)||_2\to 0.
    \end{align*}
    where $f^t_\pi = \sum \alpha_{l,k}\delta_{k,l}^\pi$.
    
    Density of $\oplus_{\pi\in Irr(\widehat{\G})}M_{n_\pi}$ in $\ell^\infty(\G)$ tells us that $||\lambda_\G(PfP)\xi_\alpha - f(P)\xi_\alpha||_2\to 0$ for all $f\in\ell^1(\G)$. Then Lemma \ref{MainThmLem} tells us $N_\omega$ is coamenable.
\end{proof}
We showed in the proof of Lemma \ref{BigLemma} that if $\widetilde{N}_P = M_P$, then coamenability of $N_\omega$ implies amenability of $\widetilde{N}_P$. This occurs in the particular case where $P$ is central, so that $\ell^\infty(\G/\H) = \widetilde{N}_P$ for some quantum subgroup $\H\leq\G$. Therefore, Lemma \ref{BigLemma} and Theorem \ref{biglemmaconverse} give us the following.
\begin{cor}\label{RelAmenIffAmenQS}
    Let $\G$ be a DQG and $\H\leq \G$ a quantum subgroup. Then $\ell^\infty(\G/\H)$ is relatively amenable if and only if it is amenable.
\end{cor}
\begin{rem}\label{LemsRemark}
\begin{enumerate}
    \item Let $\G$ be a DQG and $\H$ a quantum subgroup. While Kalantar et al{.} \cite{KKSV22} first achieved a characterization of amenability of $\H$ with relative amenability of $\ell^\infty(\G/\H)$, they did not provide any results on the {\it amenability} of {\it coideals}. In particular, Corollary \ref{RelAmenIffAmenQS} is new.
    
    \item Let us maintain the same notation as in Lemma \ref{BigLemma} and the paragraph above it. Since $P$ is central, $M_P = \widetilde{N}_P = \ell^\infty(\G/\H)$. Using the definition of coamenability of $N_\omega = L^\infty(\widehat{\H})$, it is not too difficult to prove that coamenability of $L^\infty(\widehat{\H})$ as a coideal is equivalent to coamenability of $\widehat{\H}$ as a CQG (see the proof of Corollary \ref{redcenidemps}). Thus we have given a proof that $\ell^\infty(\G/\H)$ is relatively amenable if and only if $\widehat{\H}$ is coamenable using the techniques of Vaes and Blanchard \cite{BV02} for Tomatsu's theorem \cite{Toma}, and our work on coamenable compact quasi-subgroups in Section $4.2$. Another application of Tomatsu's theorem gives us that $\ell^\infty(\G/\H)$ is relatively amenable if and only if $\H$ is amenable. Thus, we have found a different way to obtain \cite[Theorem 3.7]{KKSV22}. In their work, they use the natural action of $\H$ on $\ell^\infty(\G)$ and achieve their result by working with amenability of $\H$. In our work, we use the `dual side' of amenability, and work with coamenability of $\widehat{\H}$ instead. In Section $4.4$, we expand on this equivalence of relative amenability of $\ell^\infty(\G/\H)$ with amenability of $\H$ (see also Remark \ref{AmenRelAmenHRemark}).
\end{enumerate}
\end{rem}

\subsection{Amenability of Quantum Subgroups}
Given a DQG $\G$ and quantum subgroup $\H$, we will show that amenability and relative amenability of $\ell^\infty(\G/\H)$ characterizes amenability of $\H$. Since the group-like projection $1_\H\in\ell^\infty(\G)$ associated with $\ell^\infty(\G/\H)$ is central (cf. \cite{Kal1}), we point out that $\ell^\infty(\G/\H) = N_{1_\H} = M_{1_\H}$.

We denote the natural bimodule action of $\ell^1(\H)$ on $\ell^1(\G)$ as follows:
$$\varphi*_\H f = (\varphi\otimes f)l_\H = (\varphi\circ\sigma_\H)*f ~ \text{and} ~ f*_\H\varphi = f*(\varphi\circ\sigma_\H), ~\varphi\in \ell^1(\H), f\in \ell^1(\G),$$
where $\sigma_\H$ and $l_\H$ are defined in Section $2$.
\begin{defn}
    We will say $m\in\ell^\infty(\G)^*$ is {\bf $\H$-invariant} if
    $$\varphi(\sigma_\H\otimes m)\circ\Delta_\G = \varphi*_\H m = \varphi(1)m, ~ \varphi\in\ell^1(\H).$$
\end{defn}
We can immediately characterize the $\H$-invariant functionals as those that the annihilate the left ideals $J^1(\G,\H)$ using our preceding work.
\begin{lem}\label{DualInv}
    A non--zero functional $\mu\in \ell^\infty(\G)^*$ is $\H$-invariant if and only if $f*\mu = 0$ for all $f\in J^1(\G,\H)$.
\end{lem}
\begin{proof}
    We first claim
    \begin{align}
        \sigma_\H(\ell^\infty(\G/\H)) = \C ~ \label{Hcoset}.
    \end{align}
    Indeed, if $x\in \ell^\infty(\G/\H)$ then
    $$\Delta_\H(\sigma_\H(x)) = (\sigma_\H\otimes\sigma_\H)(\Delta_\G(x)) = \sigma_\H(x)\otimes 1,$$
    which means $\sigma_\H(x)\in \ell^\infty(\H/\H) = \C$.
    
    Now, to proceed with the proof, take $f\in J^1(\G,\H)$. Then
    \begin{align*}
        f*\mu(x) &= f\otimes \mu( \Delta_\G(x)) = f( E_{\mu}(x)) = 0
    \end{align*}
    since $E_{\mu}(x)\in \ell^\infty(\G/\H) = J^1(\G,\H)^\perp$.
    
    Conversely, because
    \begin{align*}
        f(E_{\mu}(x)) &=  f* \mu(x) = 0
    \end{align*}
    for all $f\in J^1(\G,\H)$, it follows that $E_{\mu}(x)\in \ell^\infty(\G/\H) = J^1(\G,\H)^\perp$. So, if we take $\varphi\in \ell^1(\H)$ and $x\in \ell^\infty(\G)$, then
    \begin{align*}
        \varphi *_\H \mu(x) &=  (\varphi\circ\sigma_\H)(E_\mu(x)) = \varphi(1) \overbrace{\sigma_\H(E_{\mu}(x))}^{\in\C} ~ \text{(using \eqref{Hcoset})}
        \\
        &= \varphi(1)\sigma_\H(E_{\mu}(x)) = \varphi(1) \epsilon_\G(E_{\mu}(x)) = \varphi(1)\mu(\sigma_\H(x)).
    \end{align*}
\end{proof}
Recall that $1_\H$ is the group-like projection that generates $\ell^\infty(\G/\H)$. An immediate consequence of Theorem \ref{CondRelAmen} and Lemma \ref{DualInv} is the following.
\begin{cor}\label{HInvariance1H}
    A functional $m\in \ell^1(\G)^*$ is $\H$-invariant if and only if it is left $1_\H$-invariant.
\end{cor}
\begin{rem}
    Kalantar et al{.} \cite{KKSV22} defined relative amenability of a quantum subgroup $\H$ of a DQG $\G$. They said that $\H$ is relatively amenable if there exists an $\H$-invariant state on $\ell^\infty(\G)$. They showed relative amenability of $\H$ is equivalent to amenability of $\H$ in \cite[Theorem 3.7]{KKSV22}. Corollary \ref{HInvariance1H} establishes the connection between relative amenability of $\H$ and the existence of $1_\H$-left invariant states on $\ell^\infty(\G)$.
\end{rem}
As discussed in Remark \ref{LemsRemark}, with Theorem \ref{BigLemma} and \ref{biglemmaconverse}, we have a proof of the following:
\begin{align*}
    \ell^\infty(\G/\H) ~ \text{is relatively amenable} &\iff \ell^\infty(\G/\H) ~\text{is amenable}
    \\
    &\iff \widehat{\H} ~ \text{is coamenable}
    \\
    &\iff \H~\text{is amenable (Tomatsu's theorem).}
\end{align*}
With Corollary \ref{HInvariance1H} and Theorem \ref{RelAmenPInv}, we deduce the following, which we note was proved in \cite{KKSV22}.
\begin{cor}\label{HIinvariantState}
    Let $\G$ be a DQG and $\H$ a quantum subgroup. Then $\H$ is amenable if and only if there exists an $\H$-invariant state on $\ell^\infty(\G)$.
\end{cor}
Combining the work of sections $4$ and $3$, we are able to achieve the following.
\begin{cor}\label{qsubgroupsAmen}
    Let $\G$ be a DQG and $\H$ a quantum subgroup. The following are equivalent:
    \begin{enumerate}
        \item $\H$ is amenable;
        \item $\ell^\infty(\G/\H)$ is amenable;
        \item $\ell^\infty(\G/\H)$ is relatively amenable;
        \item $J^1(\G,\H)$ has a brai;
        \item $J^1(\G,\H)$ has a brai in $\ell^1_0(\G) := \{f\in \ell^1(\G) : f(1) = 0\}$;
        \item $J^1(\G,\H)$ has a brai in $\ell^1_0(\H)$.
    \end{enumerate}
\end{cor}
\begin{proof}
    (1. $\iff$ 2. $\iff$ 3.) This is due to Lemma \ref{BigLemma}, Theorem \ref{biglemmaconverse}, and Corollary \ref{RelAmenIffAmenQS} (see the paragraph above the statement of this corollary).
    
    (1. $\implies$ 4.) This follows from Theorem \ref{IdealsAmen} after noting that $\ell^\infty(\G/\H) = M_{1_\H}$ and $J^1(\G,\H) = J^1(M_{1_\H})$ because $1_\H$ is central.
    
    (4. $\implies$ 6.) is clear.
    
    (6. $\implies$ 5.) Since $\sigma_\H$ is unital, we deduce that $\ell^1_0(\H)\circ\sigma_\H\subseteq \ell^1_0(\G)$. Then, if $(e_j)\subseteq\ell^1_0(\H)$ is a brai for $J^1(\G,\H)$, it is clear that $(e_j\circ \sigma_\H)\subseteq \ell^1_0(\G)$ is a brai for $J^1(\G,\H)$.
    
    (5. $\implies$ 3.) Let $(e_j)\subseteq\ell^1_0(\G)$ be a brai for $J^1(\G,\H)$, with weak$^*$ cluster point $\mu\in \ell^\infty(\G)^*$. Then, $f*(\epsilon_\G - \mu) = 0$ for every $f\in J^1(\G,\H)$, and so an application of Proposition \ref{DualInv} tells us $\omega = \epsilon_\G - \mu$ is a non-zero $\H$-invariant functional on $\ell^\infty(\G)$. We will manufacture an $\H$-invariant state from $\omega$. Since the decomposition $\omega = \Re(\omega) + \i\Im(\omega)$ is unique, we can assume $\omega$ is Hermitian. Also, since $\omega(1) = 1$, $\Re(\omega)(1) = 1$. Now, let $\omega = \omega_+ - \omega_1$ be the Jordan decomposition, uniquely determined so that $||\omega_+|| + ||\omega_-|| = ||\omega||$ (cf. \cite[Theorem 4.2]{Tak}). Let $\varphi\in\ell^1(\H)$ be a state. Since $\varphi*_\H\omega_+$ is positive,
    $$||\varphi*_\H \omega_+|| = (\varphi\circ\sigma_\H)*\omega_+(1) = \omega_+(1) = 1$$
    and similarly $||\varphi*_\H \omega_-|| = ||\omega_-||$. So, by uniqueness, we must have $\varphi*_\H\omega_+ = \omega_+$. Therefore, $\omega_+$ is an $\H$-invariant state. From Corollary \ref{HIinvariantState} and Theorem \ref{RelAmenPInv} we deduce that $\ell^\infty(\G/\H)$ is relatively amenable.
\end{proof}
\begin{rem}\label{AmenRelAmenHRemark}
    We must point out that Kalantar et al{.} \cite{KKSV22} achieved Corollary \ref{qsubgroupsAmen} $1\iff 3$. To obtain their result, they build an injective right $\ell^1(\G)$-module map $\ell^\infty(\H)\to \ell^\infty(\G)$, generalizing how one builds such a map $\ell^\infty(H)\to \ell^\infty(G)$ for a discrete group $G$ and subgroup $H$, using a set of representatives for the coset space $G/H$. They then used this map to establish a one-to-one correspondence between $\H$-invariant states on $\ell^\infty(\G)$ and right $\ell^1(\G)$-module maps $\ell^\infty(\G)\to \ell^\infty(\G/\H)$.
    
    In our proof, we use the correspondence between $1_\H$-invariant states on $\ell^\infty(\G)$ and right $\ell^1(\G)$-module maps $\ell^\infty(\G)\to \ell^\infty(\G/\H)$ established in Section $3$ for coideals. Then we work on the `dual side' of $\G$. We use Blanchard and Vaes' techniques in \cite{BV02} to prove that relative amenability of $\ell^\infty(\G/\H)$ is equivalent to coamenability of $\widehat{\H}$ as well as amenability of $\ell^\infty(\G/\H)$.
\end{rem}

\bibliography{AmenCoamCoideals}
\bibliographystyle{amsplain}

\end{document}